\pdfoutput=1
\documentclass[11pt]{amsart}

\usepackage[english]{babel}
\usepackage{amsmath, amssymb, fullpage}
\usepackage[activate={true,nocompatibility},final,tracking=true,kerning=true,spacing=true,factor=1100,stretch=10,shrink=10]{microtype}

\usepackage{amsthm}
\usepackage{thm-restate}
\usepackage{siunitx}

\usepackage{mathtools}
\usepackage{accents}

\usepackage{xcolor}
\usepackage{graphicx}
\usepackage{tikz}
\usepackage{tikzit}
\usetikzlibrary{arrows.meta}
\usetikzlibrary{decorations.pathreplacing}
\usetikzlibrary{decorations.markings}
\usetikzlibrary{calligraphy}

\usepackage[ocgcolorlinks]{hyperref}

\usepackage{enumitem}
\usepackage[justification=centering]{caption}
\usepackage{subcaption}

\usepackage[boxruled, longend]{algorithm2e}

\newtheorem{theorem}{Theorem}[section]
\newtheorem{theoremx}{Theorem}
\newtheorem{proposition}[theorem]{Proposition}
\newtheorem{corollary}[theorem]{Corollary}
\newtheorem{lemma}[theorem]{Lemma}

\theoremstyle{definition}
\newtheorem{definition}[theorem]{Definition}

\newtheorem{remark}[theorem]{Remark}

\newlength{\fixmidfigure}
\newenvironment{midfigure}[1][]{
  \setlength{\fixmidfigure}{\lastskip}\addvspace{-\lastskip}
  \ifstrempty{#1}{\begin{figure}}{\begin{figure}[#1]}
}{
  \end{figure}
  \addvspace{\fixmidfigure}
}
\newcommand{\N}{\mathbb{N}}
\newcommand{\Z}{\mathbb{Z}}
\newcommand{\Q}{\mathbb{Q}}
\newcommand{\R}{\mathbb{R}}
\newcommand{\C}{\mathbb{C}}

\renewcommand{\setminus}{\smallsetminus}

\newcommand{\e}{\epsilon}

\newcommand{\reduced}{N-reduced\xspace}
\newcommand{\reduction}{N-reduction\xspace}

\DeclareRobustCommand{\vo}{\accentset{\star}{v}}
\newcommand{\half}{H}
\newcommand{\ltp}{<}

\newcommand{\ltg}{<^{\star}\mkern-2mu}
\newcommand{\leqg}{\leq^\star\mkern-2mu}

\DeclareMathOperator{\Aut}{Aut}

\DeclareMathOperator{\PGL}{PGL}
\DeclareMathOperator{\SL}{SL}
\DeclareMathOperator{\PSL}{PSL}

\DeclareMathOperator{\proj}{proj}
\DeclareMathOperator{\Min}{Min}

\newcommand{\Ax}{\gamma}

\DeclarePairedDelimiter\floor{\lfloor}{\rfloor}

\DeclarePairedDelimiter{\abs}{\lvert}{\rvert}

\DeclarePairedDelimiter{\gen}{\langle}{\rangle}
\DeclarePairedDelimiterX{\inp}[2]{\langle}{\rangle}{#1, #2}

\DeclarePairedDelimiter{\stdrightbar}{}{\rvert}
\DeclarePairedDelimiter{\fullrightbar}{.}{\rvert}
\makeatletter
\newcommand\rightbar{\@ifstar{\fullrightbar*}{\stdrightbar}}
\makeatother

\newcommand\SetSymbol[1][]{{\nonscript\;}{:}{\nonscript\;}\mathopen{}}
\providecommand\given{} 
\DeclarePairedDelimiterX{\Set}[1]\{\}{\renewcommand\given{\SetSymbol[\delimsize]}#1}

\newcommand{\nset}[1]{\overline{#1}}

\microtypecontext{spacing=nonfrench}
\SetTracking{encoding={*}, shape=sc}{0}

\setlist[enumerate,1]{label=\textup{(\arabic*)}}

\tikzstyle{vertex}=[fill=black, draw=none, shape=circle, inner sep=0pt, minimum size=0.2cm]
\tikzstyle{vertex-dash}=[inner sep=0pt, decorate, decoration={markings, mark=at position 0.15 with {\draw[line width=1pt] (3pt,1pt) -- (-3pt,1pt);}}]
\tikzstyle{vertex-dash-vertical}=[inner sep=0pt, decorate, decoration={markings, mark=at position 0 with {\draw[line width=1pt] (0,3.2pt) -- (0,-2.8pt);}}]

\tikzstyle{bar}=[edge, {|-|}, dashed]
\tikzstyle{halfbar}=[edge, {|-}, dashed]
\tikzstyle{halfbar-end}=[edge, {-|}, dashed]
\tikzstyle{halfbar-solid}=[-, edge, {|-}]
\tikzstyle{bar-solid}=[-, edge, {|-|}]
\tikzstyle{edge}=[-, line width=1pt]
\tikzstyle{edge-dash}=[-, edge, dashed]
\tikzstyle{edge-bold}=[-, line width=1.5pt]
\tikzstyle{brace}=[-, edge, decorate, decoration={calligraphic brace}]
\tikzstyle{implication}=[-, {-{Stealth[length=3mm, width=3mm]}}, line width=2pt]

\colorlet{link}{blue}
\definecolor{cite}{HTML}{009900}
\hypersetup{
  colorlinks=true,
  linkcolor=link,
  citecolor=cite,
}

\definecolor{NiceGreen}{HTML}{00BB00}
\definecolor{NiceRed}{HTML}{EE0000}
\definecolor{NiceBlue}{HTML}{0033FF}

\SetKwBlock{Loop}{loop}{end loop}

\title{Recognition and constructive membership for purely hyperbolic groups acting on trees}
\author{Ari Markowitz}
\address{
Department of Mathematics, University of Auckland, 38 Princes Street, 1010 Auckland, New Zealand}
\email{ari.markowitz@auckland.ac.nz}
\begin{document}

\begin{abstract}
  We present an algorithm which takes as input a finite set $X$ of automorphisms of a simplicial tree, and outputs a generating set $X'$ of $\gen{X}$ such that either $\gen{X}$ is purely hyperbolic and $X'$ is a free basis of $\gen{X}$, or $X'$ contains a non-trivial elliptic element. As a special case, the algorithm decides whether a finitely generated group acting on a locally finite tree is discrete and free. This algorithm, which is based on  Nielsen's reduction method, works by repeatedly applying Nielsen transformations to $X$ to minimise the generators of $X'$ with respect to a given pre-well-ordering. We use this algorithm to solve the constructive membership problem for finitely generated purely hyperbolic automorphism groups of trees. We provide a \textsc{Magma} implementation of these algorithms, and report its performance.
\end{abstract}

\maketitle

\section{Introduction}
There is a rich theory of groups acting on geometric spaces. Of interest are groups which are discrete and free; examples are Schottky subgroups of $\PSL_2(\R)$ and $\PSL_2(\C)$ acting on hyperbolic space \cite{beardon, maskit}, and of $\PSL_2(K)$ for a non-archimedean local field $K$ acting on the Bruhat-Tits tree \cite{lubotzky, mumford}. For a topological group $G$ acting on such a space, some natural problems are the following:
\begin{enumerate}
  \item
  Decide whether $G$ is discrete and free, and if so find a free basis for $G$.
  \item
  \emph{The constructive membership problem.} Suppose $G$ is a subgroup of a group $H$. Given $g \in H$, decide whether $g \in G$, and if so write $g$ as a word in a specified generating set of $G$.
\end{enumerate}
The first problem was solved for 2-generator subgroups of $\PSL_2(\R)$ by Purzitsky \cite{purzitsky}.
The second problem was solved for discrete free 2-generator subgroups of $\PSL_2(\R)$ by Eick, Kirschmer, and Leedham-Green \cite{kirschmer}, and was later generalised to discrete 2-generator subgroups of $\PSL_2(\R)$ by Kirschmer and Rüther \cite{kirschmer-discrete}. For larger numbers of generators, the problems remain open.

Conder \cite{conder} solved (1) for 2- and 3-generator groups of automorphisms of locally finite trees, and conjectures an algorithm for every finite number of generators \cite[Conjecture 2.3]{conder}. In this paper, we solve (1) for all finitely generated groups of automorphisms of locally finite trees. More generally, we provide an algorithm to decide whether a finitely generated group of automorphisms of a simplicial tree is purely hyperbolic (that is, every vertex has trivial stabiliser). We also solve (2) for finitely generated purely hyperbolic subgroups of automorphism groups of trees.

The standard approach to these problems, as done in \cite{conder, kirschmer, gilman}, uses the interactions between translation axes and translation lengths of the generators. This is difficult to generalise to higher numbers of generators, since the number of possible interactions grows rapidly.

In Section \ref{sec:reduction} we prove the following theorem:
\begin{theoremx}\label{thm:main}
  Let $T$ be a simplicial tree. Let $X$ be a finite subset of $\Aut(T)$ generating a group $G$. There exists an algorithm that, given $X$, outputs a basis $X'$ of $G$ that either contains an elliptic element or is a free basis for $G$.
\end{theoremx}
We use an analogue of Nielsen's method of finding a free basis for a finitely generated subgroup of a free group \cite[Chapter I.2]{lyndon-schupp}. We define a \emph{strongly \reduced} basis, and show that if $G$ is purely hyperbolic (in particular, if $T$ is locally finite and $G$ is discrete and free), then a strongly \reduced basis exists. This provides an algorithmic version of a theorem of Weidmann \cite{weidmann} which, while non-constructive, uses similar methods to show that such a basis  exists.

As a consequence, we obtain the following:
\begin{theoremx}\label{thm:mainalgo}
  Let $T$ be a locally finite simplicial tree. Let $X$ be a finite subset of $\Aut(T)$ generating a group $G$. There exists an algorithm that, given $X$, decides whether $G$ is discrete and free.
\end{theoremx}
In Section \ref{sec:strongly-reduce} we prove that if $G$ is purely hyperbolic, then it has a unique strongly \reduced basis.
In Section \ref{sec:constructive-membership} we associate to $G$ a fundamental domain $\Gamma(G)$, and provide an algorithm that takes as input a vertex $v$ of $T$ and outputs the unique $g \in G$ such that $gv\in \Gamma(G)$. With this we prove the following:
\begin{restatable}{theoremx}{conMem}\label{thm:constructive-membership}
  Every finitely generated purely hyperbolic subgroup of $\Aut(T)$ has solvable constructive membership problem.
\end{restatable}

In Section \ref{sec:implementation}, we discuss our implementation of these algorithms in \textsc{Magma} \cite{magma} for finitely generated subgroups of $\PGL_2(K)$ acting on the Bruhat-Tits tree where $K$ is a $p$-adic field, and report its performance.

\subsection{Related work}
The application of Nielsen reduction to trees is not new. Weidmann \cite{weidmann} uses similar methods to prove results relating to finitely generated groups acting on trees. In particular, he provides conditions under which a finite set of automorphisms of a tree generates an amalgamated product where each factor is either free or has a global fixed point \cite[Theorem 2]{weidmann}.

In another generalisation, Kapovich and Weidmann \cite{kapovich} prove that a finite generating set of a group $G$ acting on a $\delta$-thin hyperbolic space is Nielsen equivalent to a generating set of $G$ that is either a free basis for $G$ or contains a generator of ``small" translation length.

Both results provide a non-constructive version of Theorem \ref{thm:main} as a special case. The method works by iteratively replacing a generator with an element of $G$ that can be obtained via Nielsen transformations and is smaller with respect to a given pre-order, until no further replacements are possible. However, the methods in these papers are non-constructive, as they do not provide a bound on the word length of such an element. We show that if such a replacement exists, then one exists with word length 2. This allows us to obtain a practical algorithm to find a free basis, and to solve the constructive membership problem.

There is also a generalisation of Nielsen's methods to groups with a length function, as detailed by Hoare \cite{hoare} and by Lyndon and Schupp \cite[Chapter I.9]{lyndon-schupp}. Depending on the axioms chosen, $\Aut(T)$ can be endowed with a length function. However, to our knowledge none of the results on length functions specialise to the results in this paper.

\section{Nielsen reduction}\label{sec:reduction}
For the remainder of this paper we fix the following notation:
\begin{itemize}
  \item $T$ is a simplicial tree with vertex set $V(T)$. We identify $T$ with its geometric realisation, so that a point $w \in T$ may be a vertex or lie on an edge.
  \item $X$ is a finite subset of $\Aut(T)$ generating a group $G$ equipped with the compact-open topology \cite{garrido}.
  \item $\vo$ is a distinguished vertex of $T$. It may be chosen arbitrarily, but once chosen it remains fixed.
  \item If $u$ and $w$ are vertices of $T$, then $[u, w]$ is the unique path (without backtracking) from $u$ to $w$. We identify this path with the corresponding line segment in the geometric realisation of $T$.
  \item Given a path $p$ on $T$ and $C \subseteq T$, if $p \cap C$ is the geometric realisation of a subpath $q$ of $p$, then we identify $p \cap C$ with $q$.
  \item If $p$ is a path or a walk (that is, a path with possible backtracking) on $T$, then $\abs{p}$ is the length of $p$. For a walk, this is the number of terms in the corresponding sequence of edges.
  \item $X^- = \Set{g^{-1} \given g \in X}$ and $X^\pm = X \cup X^-$.
  \item Given $g \in G$, define $\abs{g} = d(\vo, g\vo)$, where $d$ is the graph distance on $T$.
  \item Given $g \in G$, the \emph{translation length} of $g$ is $l(g) = \min\Set{d(w, gw) \given w \in T}$.
    The \emph{minimum translation set} of $g$ is $\Min(g) = \Set{w \in T \given d(w, gw) = l(g)}$.
\end{itemize}
Recall that $g$ is \emph{elliptic} if $l(g) = 0$, and \emph{hyperbolic} if $l(g) > 0$. Elliptic elements are typically distinguished from \emph{inversions} which invert an edge \cite{garrido}. By our definition, inversions are elliptic: If $g$ inverts an edge $e$, then $l(g) = 0$ and $\Min(g)$ is the midpoint of $e$. If all nontrivial elements of $G$ are hyperbolic, then $G$ is \emph{purely hyperbolic}.

This section closely follows Nielsen's proof of his Subgroup Theorem, which states that a finitely generated subgroup of a free group is free \cite[Chapter I.2]{lyndon-schupp}. In particular, Nielsen uses cancellation of words in a free group; we use cancellation of paths on a tree. The use of conditions N1, N2, and N3 (see Definition \ref{def:n-reduced}), and the pre-well-ordering defined on $G$ (see Definition \ref{def:more-orderings}), remains essentially the same. In our case, an added complication is the possible existence of elliptic elements.

\begin{definition}\label{def:delta}
  If $g, h \in G$, then $\delta(g, h) = (\abs{g} + \abs{h} - \abs{g^{-1}h}$)/2. 
\end{definition}
Note that $\delta(g, h) = \delta(h, g)$, but $\delta(g, h) \neq \delta(g^{-1}, h^{-1})$ in general.
\begin{proposition}\label{prop:delta-is-intersection}
  For all $g, h \in G$, $\delta(g, h) = \abs{[\vo, g\vo] \cap [\vo, h\vo]}$.
\end{proposition}
\begin{proof}
  Let $w$ be the vertex of $T$ such that $[\vo, g\vo] \cap [\vo, h\vo] = [\vo, w]$. Now $[g\vo, h\vo] = [g\vo, w] \cup [w, h\vo]$; see Figure \ref{fig:paths}. We deduce that
  \begin{align*}
    \delta(g, h) &= (d(\vo, g\vo) + d(\vo, h\vo) - d(\vo, g^{-1}h\vo))/2 \\
    &= (d(\vo, g\vo) + d(\vo, h\vo) - d(g\vo, h\vo))/2 \\
    &= (d(\vo, g\vo) + d(\vo, h\vo) - d(w, g\vo) - d(w, h\vo))/2 \\
    &= d(\vo, w) \\
    &= \abs{[\vo, g\vo] \cap [\vo, h\vo]}. \qedhere
  \end{align*}
\end{proof}
\begin{figure}[ht]
  \centering
  \begin{tikzpicture}
	\begin{pgfonlayer}{nodelayer}
		\node [style=vertex, label={below right:$g\vo$}] (0) at (4, -4) {};
		\node [style=vertex, label={above right:$h\vo$}] (1) at (4, 4) {};
		\node [style=vertex, label={[label distance=-0.05cm]above:$w$}] (2) at (0, 0) {};
		\node [style=vertex, label={left:$\vo$}] (4) at (-8, 0) {};
		\node [style=none] (6) at (-8, 2) {};
		\node [style=none] (7) at (0, 2) {};
		\node [style=none] (8) at (-8, -1) {};
		\node [style=none] (9) at (-0.5, -1) {};
		\node [style=none] (10) at (3.25, -4.75) {};
		\node [style=none] (11) at (-8, 1) {};
		\node [style=none] (12) at (-0.5, 1) {};
		\node [style=none] (13) at (3.25, 4.75) {};
		\node [style=none] (14) at (4.75, -3.25) {};
		\node [style=none, label={[label distance=15pt]right: $\abs{g^{-1}h}$}] (15) at (1.5, 0) {};
		\node [style=none] (16) at (4.75, 3.25) {};
	\end{pgfonlayer}
	\begin{pgfonlayer}{edgelayer}
		\draw [style=edge] (0) to (2);
		\draw [style=edge] (1) to (2);
		\draw [style=edge] (2) to (4);
		\draw [style=bar] (6.center) to node [auto] {$\delta(g, h)$} (7.center);
		\draw [style=halfbar] (8.center) to (9.center);
		\draw [style=halfbar] (10.center) to node [auto] {$\abs{g}$} (9.center);
		\draw [style=halfbar] (11.center) to (12.center);
		\draw [style=halfbar] (13.center) to node [midway, above left] {$\abs{h}$} (12.center);
		\draw [style=halfbar] (14.center) to (15.center);
		\draw [style=halfbar] (16.center) to (15.center);
	\end{pgfonlayer}
\end{tikzpicture}
  \caption{$\delta(g, h)$ is the length of $[\vo, g\vo] \cap [\vo, h\vo]$}
  \label{fig:paths}
\end{figure}
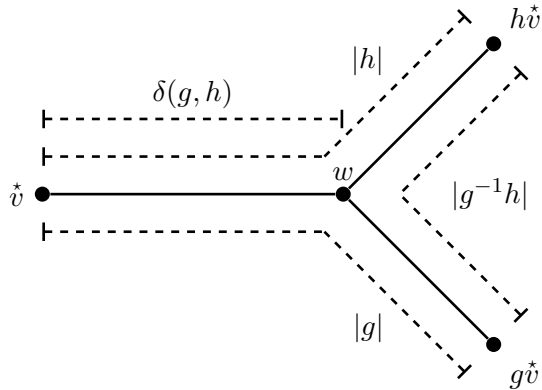
\begin{proposition}\label{prop:translation-length-formula}
  For all $g \in G$, $l(g) = \abs{g^2}-\abs{g} = \abs{g} - 2\delta(g^{-1}, g)$.
\end{proposition}
\begin{proof}
  By \cite[Chapter I.6.4, Propositions 23 and 24]{serre-trees}, the midpoint of $[\vo, g\vo]$ lies in $\Min(g)$. Let $w = \proj_{\Min(g)}(\vo)$. Now,
  \begin{align*}
    d(\vo, g^2\vo) &= d(\vo, w) + d(w, g^2w) + d(g^2w, g^2\vo) \\
    &= d(\vo, w) + 2d(w, gw) + d(g\vo, gw) \\
    &= d(\vo, g\vo) + l(g).
  \end{align*}
  This proves the first equality; the second follows from Definition \ref{def:delta}.
\end{proof}
Let $g = a_1 a_2 \dots a_n$ be a reduced word in $X$. Each $a_i$ may be identified with the path $p_i = [b_{i}\vo, b_{i} a_i \vo]$, where $b_i = a_1 \dots a_{i-1}$. We say that $p_i$ is the \emph{path of} $a_i$ in $g$.
Note that a walk from $\vo$ to $g\vo$ is formed by the concatenation of the $p_i$, each of which is isometric to $[\vo, a_i\vo]$. If a subpath $q$ of $p_i$ of nonzero length is disjoint from the interior of $[\vo, g\vo]$, then it is \emph{cancelled} in $g$. If $q$ is contained in $[\vo, g\vo]$, then it is \emph{uncancelled} in $g$.
\begin{proposition}\label{prop:delta-is-intersect-2}
  If $x, y \in X^\pm$, then $[\vo, x\vo] \cup [x\vo, xy\vo] = [\vo, xy\vo] \cup p$, where $p = [\vo, x\vo] \cap [x\vo, xy\vo]$ is a (possibly empty) path of length $\delta(x^{-1}, y)$.
\end{proposition}
\begin{proof}
  By Proposition \ref{prop:delta-is-intersection}, $\delta(x^{-1}, y) = \abs{[\vo, x^{-1}\vo] \cap [\vo, y\vo]} = \abs{[\vo, x\vo] \cap [x\vo, xy\vo]}$. We see from Figure \ref{fig:product-cancel} that $[\vo, x\vo] \cup [x\vo, xy\vo] = [\vo, xy\vo] \cup p$.
\end{proof}
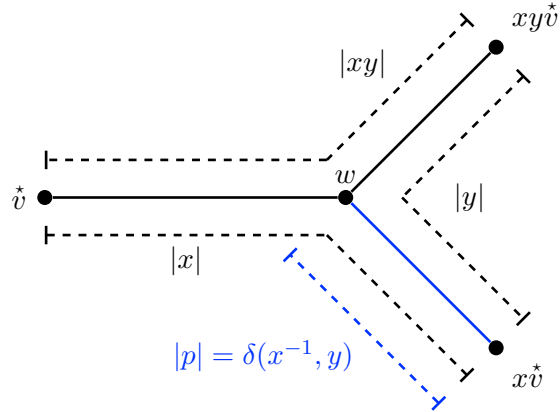
\begin{figure}[ht]
  \centering
  \begin{tikzpicture}
	\begin{pgfonlayer}{nodelayer}
		\node [style=vertex, label={below right:$x\vo$}] (0) at (4, -3.25) {};
		\node [style=vertex, label={above right:$xy\vo$}] (1) at (4, 4.75) {};
		\node [style=vertex, label={[label distance=-0.05cm, above]:$w$}] (2) at (0, 0.75) {};
		\node [style=vertex, label={left:$\vo$}] (4) at (-8, 0.75) {};
		\node [style=none] (6) at (2.5, -4.75) {};
		\node [style=none] (7) at (-1.5, -0.75) {};
		\node [style=none] (8) at (-8, -0.25) {};
		\node [style=none] (9) at (-0.5, -0.25) {};
		\node [style=none] (10) at (3.25, -4) {};
		\node [style=none] (11) at (-8, 1.75) {};
		\node [style=none] (12) at (-0.5, 1.75) {};
		\node [style=none] (13) at (3.25, 5.5) {};
		\node [style=none] (14) at (4.75, -2.5) {};
		\node [style=none, label={[label distance=15pt]right: $\abs{y}$}] (15) at (1.5, 0.75) {};
		\node [style=none] (16) at (4.75, 4) {};
	\end{pgfonlayer}
	\begin{pgfonlayer}{edgelayer}
		\draw [style=edge, color=NiceBlue] (0) to (2);
		\draw [style=edge] (1) to (2);
		\draw [style=edge] (2) to (4);
		\draw [style=bar, color=NiceBlue] (6.center) to node [midway, below left] {$\abs{p} = \delta(x^{-1}, y)$} (7.center);
		\draw [style=halfbar] (8.center) to node [midway, below] {$\abs{x}$} (9.center);
		\draw [style=halfbar-end] (9.center) to (10.center);
		\draw [style=halfbar-end] (12.center) to (11.center);
		\draw [style=halfbar] (13.center) to node [midway, above left] {$\abs{xy}$} (12.center);
		\draw [style=halfbar-end] (15.center) to (14.center);
		\draw [style=halfbar] (16.center) to (15.center);
	\end{pgfonlayer}
\end{tikzpicture}
  \caption{The path $p$ is cancelled in $xy$, while $[\vo, w]$ and $[w, xy\vo]$ are uncancelled}
  \label{fig:product-cancel}
\end{figure}
\begin{definition}\label{def:n-reduced}
  Following the notation of  \cite[Chapter I.2]{lyndon-schupp}, $X$ is \emph{Nielsen-reduced} (or \emph{\reduced} for short) if it has no non-trivial elliptic elements and the following are satisfied:
  \begin{enumerate}
    \item[N1.] $X \cap X^- = \varnothing$.
    \item[N2.] For all $x, y \in X^{\pm}$, if $x \neq y^{-1}$, then $\abs{xy} \geq \max\{\abs{x}, \abs{y}\}$.
    \item[N3.] For all $x, y, z \in X^{\pm}$, if $x \neq y^{-1}$ and $y \neq z^{-1}$, then $\abs{xyz} > \abs{x} + \abs{z} - \abs{y}$.
  \end{enumerate}
\end{definition}
\begin{remark}
  We use a stronger definition of N1 than \cite[Chapter I.2]{lyndon-schupp}; that definition allows both $x$ and $x^{-1}$ to be in $X$.
\end{remark}
\begin{lemma}\label{lem:no-overlap-conditions}
  Let $x, y, z \in X^\pm$. Let $p = [\vo, x\vo] \cap [x\vo, xy\vo]$ and $q = [x\vo, xy\vo] \cap [xy\vo, xyz\vo]$.
  Let $\Delta = \abs{y} - \delta(x^{-1}, y) - \delta(y^{-1}, z)$. If $\Delta > 0$, then $p$ and $q$ do not intersect and $d(p, q) = \Delta$.
  If $\Delta \leq 0$, then $p \cap q$ is a path of length $-\Delta$.
\end{lemma}
\begin{proof}
  Let $\phi$ be the isometric embedding of $[x\vo, xy\vo]$ into $\R$ such that $\phi(x\vo) = 0$ and $\phi(xy\vo) = \abs{y}$. We identify the image of a path under $\phi$ with the image of the points it contains. By Proposition \ref{prop:delta-is-intersect-2}, $\phi(p) = [0, \delta(x^{-1}, y)]$ and $\phi(q) = [\abs{y}-\delta(y^{-1}, z), \abs{y}]$. The result follows by computing either $d(\phi(p), \phi(q))$ or the length of $\phi(p) \cap \phi(q)$.
\end{proof}
\begin{definition}
  We define a similar set of conditions to N2 and N3 on $X$:
  \begin{enumerate}
    \item[N2$'$.] For all $x, y \in X^{\pm}$, if $x \neq y^{-1}$, then $\delta(x^{-1}, y) \leq \min\{\abs{x}/2, \abs{y}/2\}$.
    \item[N3$'$.] For all $x, y, z \in X^{\pm}$, if $x \neq y^{-1}$ and $y \neq z^{-1}$, then $\delta(x^{-1}, y) + \delta(y^{-1}, z) < \abs{y}$.
  \end{enumerate}
\end{definition}
\begin{proposition}\label{prop:equivalent-axioms}\leavevmode
  \begin{enumerate}
    \item $\textup{N2}$ is equivalent to $\textup{N2}'$.
    \item If $\textup{N2}$ holds, then $\textup{N3}$ is equivalent to $\textup{N3}'$.
  \end{enumerate}
\end{proposition}
\begin{proof}
  Let $x, y, z \in X^{\pm}$.
  \begin{enumerate}
    \item
    Observe that $\delta(x^{-1}, y) = (\abs{x} + \abs{y} - \abs{xy})/2 \leq \min\{\abs{x}/2, \abs{y}/2\}$ if and only if $\abs{xy} \ge \max\{\abs{x}, \abs{y}\}$. See Figure \ref{fig:axioms} (a).
    \item
    Let $\Delta = \abs{y} - \delta(x^{-1}, y) - \delta(y^{-1}, z)$. By N2$'$, $\Delta \geq 0$. By Lemma \ref{lem:no-overlap-conditions}, we have the situation in Figure \ref{fig:axioms} (b): Namely, $\abs{xyz} = \abs{x} + \abs{z} - \abs{y} + \Delta$. Thus N3 holds if and only if $\Delta > 0$, which in turn is true if and only if N3$'$ holds. \qedhere
  \end{enumerate}
\end{proof}
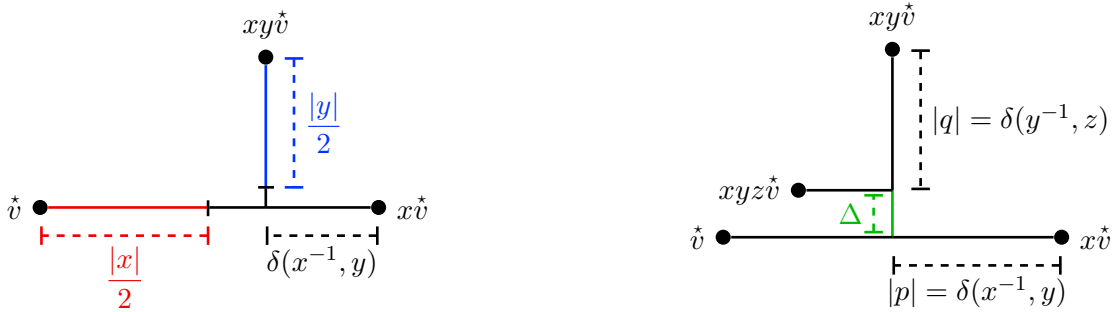
\begin{figure}[ht]
  \centering
  \begin{subfigure}{0.45\textwidth}
    \centering
    \begin{tikzpicture}
	\begin{pgfonlayer}{nodelayer}
		\node [style=vertex, label={right:$x\vo$}] (0) at (4.5, 0) {};
		\node [style=vertex, label={above:$xy\vo$}] (1) at (1.5, 4) {};
		\node [style=vertex, label={left:$\vo$}] (4) at (-4.5, 0) {};
		\node [style=none] (14) at (1.5, 0) {};
		\node [style=vertex-dash] (15) at (1.5, 0.5) {};
		\node [style=vertex-dash-vertical] (16) at (0, 0) {};
		\node [style=none] (17) at (-4.5, -0.75) {};
		\node [style=none] (18) at (0, -0.75) {};
		\node [style=none] (19) at (2.25, 0.5) {};
		\node [style=none] (20) at (2.25, 4) {};
		\node [style=none] (21) at (4.5, -0.75) {};
		\node [style=none] (22) at (1.5, -0.75) {};
	\end{pgfonlayer}
	\begin{pgfonlayer}{edgelayer}
		\draw [style=edge] (14.center) to (0);
		\draw [style=edge] (14.center) to (15.center);
		\draw [style=edge, color=NiceBlue] (15.center) to (1);
		\draw [style=edge] (14.center) to (16.center);
		\draw [style=edge, color=NiceRed] (16.center) to (4);
		\draw [style=bar, color=NiceRed] (18.center) to node [auto] {$\dfrac{\abs{x}}{2}$} (17.center);
		\draw [style=bar, color=NiceBlue] (20.center) to node [auto] {$\dfrac{\abs{y}}{2}$} (19.center);
		\draw [style=bar] (21.center) to node [auto] {$\delta(x^{-1}, y)$} (22.center);
	\end{pgfonlayer}
\end{tikzpicture}
    \caption{\centering N2$'$: At least half of the paths of both $x$ and $y$ are uncancelled in $xy$}
    \label{fig:a2}
  \end{subfigure}%
  \hfill
  \begin{subfigure}{0.45\textwidth}
    \centering
    \begin{tikzpicture}
	\begin{pgfonlayer}{nodelayer}
		\node [style=vertex, label={right:$x\vo$}] (0) at (4.5, 0) {};
		\node [style=vertex, label={above:$xy\vo$}] (1) at (0, 5) {};
		\node [style=vertex, label={left:$\vo$}] (4) at (-4.5, 0) {};
		\node [style=none] (16) at (0, 0) {};
		\node [style=none] (19) at (0.75, 1.25) {};
		\node [style=none] (20) at (0.75, 5) {};
		\node [style=none] (21) at (4.5, -0.75) {};
		\node [style=none] (22) at (0, -0.75) {};
		\node [style=none] (23) at (0, 1.25) {};
		\node [style=vertex, label={left:$xyz\vo$}] (24) at (-2.5, 1.25) {};
		\node [style=none] (25) at (-0.5, 0.1) {};
		\node [style=none] (26) at (-0.5, 1.15) {};
	\end{pgfonlayer}
	\begin{pgfonlayer}{edgelayer}
		\draw [style=edge, color=NiceGreen] (16.center) to (23.center);
		\draw [style=edge] (16.center) to (4);
		\draw [style=bar] (20.center) to node [auto] {$\abs{q} = \delta(y^{-1}, z)$} (19.center);
		\draw [style=bar] (21.center) to node [auto] {$\abs{p} = \delta(x^{-1}, y)$} (22.center);
		\draw [style=edge] (16.center) to (0);
		\draw [style=edge] (23.center) to (24);
		\draw [style=bar, color=NiceGreen] (25.center) to node [auto] {$\Delta$} (26.center);
		\draw [style=edge] (23.center) to (1);
	\end{pgfonlayer}
\end{tikzpicture}
    \caption{\centering N3$'$: The path of $y$ has a subpath of length $\Delta$ uncancelled in $xyz$}
    \label{fig:a3}
  \end{subfigure}
  \caption{Examples of conditions N2$'$ and N3$'$}
  \label{fig:axioms}
\end{figure}
\begin{remark}
The conditions N2$'$ and N3$'$ can be interpreted geometrically: N2$'$ states that at least half of the path of $x$ (similarly $y$) must be uncancelled in $xy$; N3$'$ states that there must be a subpath of the path of $y$ uncancelled in $xyz$.
\end{remark}
\begin{definition}
  A \emph{Nielsen transformation} of $X$ is a composition of the following operations:
  \begin{enumerate}
    \item Remove some $g \in X$, where both $g$ and $g^{-1}$ are in $X$.
    \item Replace some $g \in X$ with $g^{-1}$.
    \item Replace some $g \in X$ by $g^{\e_1} h^{\e_2}$ or $h^{e_2} g^{e_1}$, where $h \in X \setminus \{g\}$ and $\e_1, \e_2 \in \{1, -1\}$.
  \end{enumerate}
  This definition of a Nielsen transformation is equivalent to the usual definition (as seen in \cite[Chapter I.2]{lyndon-schupp}), but is more useful for this paper.
\end{definition}
We sometimes refer to a replacement of $g$ by $h$ where $g \in X^{\pm}$. This denotes either a replacement of $g$ by $h$ or of $g^{-1}$ by $h^{-1}$, depending on whether $g$ or $g^{-1}$ is in $X$. Note that if $X'$ is obtained from $X$ via a Nielsen transformation, then $\gen{X'} = \gen{X}$.

The following lemma demonstrates a \emph{local-to-global} phenomenon that can be similarly observed in Nielsen's algorithm: A path cancelled in a word $g$ in an \reduced set is isometric to a path cancelled in a subword of length 2. This allows us to express $\abs{g}$ in terms of the lengths of cancellations between consecutive terms.
\begin{lemma}\label{lem:no-overlap}
  Suppose $X$ is \reduced. Let $g = a_1 a_2 \dots a_n$ be a reduced word in $X$. Define $g_i = a_1 \dots a_i$.
  \begin{enumerate}
    \item If $n \geq 2$, then $[\vo, g_{n-2}\vo] \cap [g_{n-1}\vo, g\vo] = \varnothing$.
    \item If $n \geq 2$, then $[\vo, g_{n-1} \vo] \cap [g_{n-1} \vo, g\vo] = [g_{n-2} \vo, g_{n-1} \vo] \cap [g_{n-1} \vo, g\vo]$, and in particular $\delta(g_{n-1}^{-1}, a_n) = \delta(a_{n-1}^{-1}, a_n)$.
    \item $\abs{g} = \sum\limits_{i = 1}^n \abs{a_i} - 2\sum\limits_{i = 1}^{n-1} \delta(a_i^{-1}, a_{i+1})$.
  \end{enumerate}
\end{lemma}
\begin{proof}
  We proceed by induction on $n$. The cases $n \leq 2$ are trivial. Suppose $n>2$, as shown in Figure \ref{fig:chain}. By (2) and N3$'$,
  \[
    \abs{a_{n-1}} - \delta(g_{n-2}^{-1}, a_{n-1}) - \delta(a_{n-1}^{-1}, a_n) = \abs{a_{n-1}} - \delta(a_{n-2}^{-1}, a_{n-1}) - \delta(a_{n-1}^{-1}, a_n) > 0.
  \] By Lemma \ref{lem:no-overlap-conditions},
  \[
    [\vo, a_{n-2}\vo] \cap [a_{n-2}a_{n-1}\vo, a_{n-2}a_{n-1}a_n\vo] = \varnothing.
  \]
  If $n=3$, then this proves (1). If $n>3$, then by the induction hypothesis on (2),
  \begin{align*}
    &[\vo, g_{n-2}\vo] \cap [g_{n-2}\vo, g_{n-1}\vo] \cap [g_{n-1}\vo, g\vo] \\
    &= [g_{n-3}\vo, g_{n-2}\vo] \cap [g_{n-2}\vo, g_{n-1}\vo] \cap [g_{n-1}\vo, g\vo] \\ &= \varnothing.
  \end{align*}
  We therefore have the situation shown in Figure \ref{fig:chain}, proving (1) and (2). By Definition \ref{def:delta},
  \[
    \abs{g} = \abs{g_{n-1}} + \abs{a_n} - 2\delta(g_{n-1}^{-1}, a_n) = \abs{g_{n-1}} + \abs{a_n} - 2\delta(a_{n-1}^{-1}, a_n),
  \]
  from which (3) follows by induction.
\end{proof}
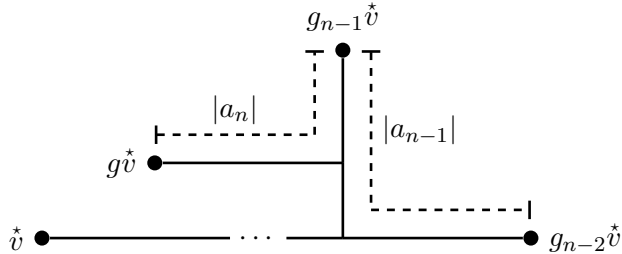
\begin{figure}[ht]
  \centering
  \begin{tikzpicture}
	\begin{pgfonlayer}{nodelayer}
		\node [style=vertex, label={right:$g_{n-2}\vo$}] (0) at (5, 0) {};
		\node [style=vertex, label={left:$\vo$}] (4) at (-8, 0) {};
		\node [style=none] (16) at (0, 0) {};
		\node [style=none] (23) at (0, 2) {};
		\node [style=vertex, label={above:$g_{n-1}\vo$}] (24) at (0, 5) {};
		\node [style=vertex, label={left:$g\vo$}] (27) at (-5, 2) {};
		\node [style=none] (31) at (-2.25, 0) {$\cdots$};
		\node [style=none] (32) at (-1.5, 0) {};
		\node [style=none] (33) at (-3, 0) {};
		\node [style=none] (34) at (5, 0.75) {};
		\node [style=none] (35) at (0.75, 0.75) {};
		\node [style=none] (36) at (0.75, 5) {};
		\node [style=none] (37) at (-0.75, 5) {};
		\node [style=none] (38) at (-0.75, 2.75) {};
		\node [style=none] (39) at (-5, 2.75) {};
	\end{pgfonlayer}
	\begin{pgfonlayer}{edgelayer}
		\draw [style=edge] (16.center) to (0);
		\draw [style=edge] (23.center) to (24);
		\draw [style=edge] (16.center) to (23.center);
		\draw [style=edge] (16.center) to (32.center);
		\draw [style=edge] (33.center) to (4);
		\draw [style=edge] (23.center) to (27);
		\draw [style=halfbar-end] (35.center) to (34.center);
		\draw [style=halfbar] (36.center) to node [auto] {$\abs{a_{n-1}}$} (35.center);
		\draw [style=halfbar-end] (38.center) to (37.center);
		\draw [style=halfbar] (39.center) to node [auto] {$\abs{a_{n}}$} (38.center);
	\end{pgfonlayer}
\end{tikzpicture}
  \caption{The path of $g = a_1 \dots a_n$.}
  \label{fig:chain}
\end{figure}
From this we obtain a key result:
\begin{proposition}\label{prop:reduced-implies-free}
  If $X$ is \reduced, then $G$ is purely hyperbolic, and $X$ is a free basis for $G$.
\end{proposition}
\begin{proof}
  Let $g = a_1 \dots a_n$ be a non-trivial reduced word in $X$. Since $X$ contains no elliptic elements, we may assume that $n>1$. Define $a_0 = a_n$ and $a_1 = a_{n+1}$. By Lemma \ref{lem:no-overlap} (3) and N3$'$,
  \begin{align*}
    \abs{g^2} - \abs{g}
    &= \sum\limits_{i=1}^n \abs{a_i} - 2\sum\limits_{i=1}^{n-1} \delta(a_i^{-1}, a_{i+1}) - 2\delta(a_n^{-1}, a_1) \\
    &= \sum\limits_{i=1}^n \left(\abs{a_i} - \delta(a_{i-1}^{-1}, a_{i}) - \delta(a_i^{-1}, a_{i+1})\right) \\
    &> 0.
  \end{align*}
  By Lemma \ref{prop:translation-length-formula}, $g$ is hyperbolic and therefore represents a non-trivial element of $G$. By \cite[Chapter I, Proposition 1.9]{lyndon-schupp}, $X$ is a free basis for $G$.
\end{proof}
In the remainder of this section we describe and prove the correctness of the reduction algorithm that obtains an \reduced set. A straightforward approach to such an algorithm is to apply Nielsen transformations to $X$, replacing $x$ with $xy$ such that $\abs{xy} < \abs{x}$ when possible. However, a ``tie-breaker'' is needed when $\abs{x} = \abs{xy}$, in particular when $X$ satisfies N2 but violates N3. To this end, we construct a pre-order on $G$ such that for every word $xyz$ in $X^\pm$ in which the path of $y$ is fully cancelled, either the replacement of $x$ by $xy$ or of $z$ by $yz$ will reduce the respective generator under this pre-order.
\begin{definition}
  If $g \in G$, then the \emph{initial half} of $g$ is the subpath $\half(g)$ of $[\vo, g\vo]$ with initial vertex $\vo$ of length $\floor*{\abs{g}/2}$.
\end{definition}
\begin{lemma}\label{lem:same-initial-segment}
  Let $x, y \in X$. If $\delta(x^{-1}, y) \leq \min\{\abs{x}/2, \abs{y}/2\}$ and $\abs{xy} = \abs{x}$, then $\half(x) = \half(xy)$.
\end{lemma}
\begin{proof}
  Let $w$ be a vertex of $T$ such that $[w, x\vo] = [\vo, x\vo] \cap [x\vo, xy\vo]$, as in Figure \ref{fig:half-path}. Then
  \[
    d(\vo, w) = \abs{x} - \delta(x^{-1}, y) \ge \abs{x}/2 = \abs{xy}/2,
  \]
  so $\half(x) = \half(xy) \subseteq [\vo, w]$.
\end{proof}
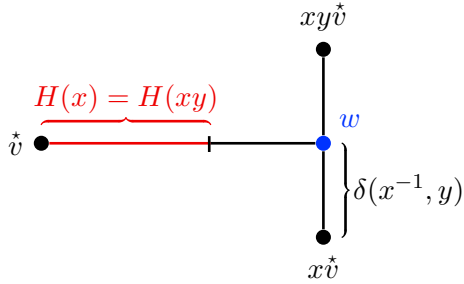
\begin{figure}[ht]
  \centering
  \begin{tikzpicture}
	\begin{pgfonlayer}{nodelayer}
		\node [style=vertex, label={below:$x\vo$}] (0) at (0, -2.5) {};
		\node [style=vertex, label={above:$xy\vo$}] (1) at (0, 2.5) {};
		\node [style=vertex, label={left:$\vo$}] (4) at (-7.5, 0) {};
		\node [style=vertex, color=NiceBlue, label={[text=NiceBlue]above right:$w$}] (16) at (0, 0) {};
		\node [style=vertex-dash-vertical] (40) at (-3, 0) {};
		\node [style=none] (41) at (-3, 0.5) {};
		\node [style=none] (42) at (-7.5, 0.5) {};
		\node [style=none] (43) at (0.5, -2.5) {};
		\node [style=none] (44) at (0.5, 0) {};
	\end{pgfonlayer}
	\begin{pgfonlayer}{edgelayer}
		\draw [style=edge] (16) to (0);
		\draw [style=edge] (1) to (16);
		\draw [style=edge] (16) to (40.center);
		\draw [style=edge, color=NiceRed] (40.center) to (4);
		\draw [style=brace, pen colour=NiceRed, color=NiceRed] (42.center) to node [auto] {$\half(x) = \half(xy)$} (41.center);
		\draw [style=brace] (44.center) to node [auto] {$\delta(x^{-1}, y)$} (43.center);
	\end{pgfonlayer}
\end{tikzpicture}
  \caption{The initial half of $x$ and $xy$}
  \label{fig:half-path}
\end{figure}
Let $P$ be the set of paths in $T$ with initial vertex $\vo$, let $P_0 = \{(\vo)\}$, and for each $i \in \N$ let $P_i$ be the set of paths in $P$ of length $i$. We choose a lexicographic well-ordering $\ltp$ on $P$, via the following construction:
Let $<_0$ be the trivial ordering on $P_0$. Let $v_1 = w_1 = \vo$. For each $i \in \N$, choose a well-ordering $<_i$ on $P_i$ such that $(v_1, \dots, v_{i+1}) <_i (w_1, \dots, w_{i+1})$ if $(v_1, \dots, v_{j+1}) <_j (w_1, \dots, w_{j+1})$ for some $j < i$.
\begin{definition}\label{def:well-orderings}
For $p, q \in P$, define $p \ltp q$ if either $\abs{p} < \abs{q}$, or $\abs{p} = \abs{q}$ and $p <_{\abs{p}} q$.
\end{definition}
\begin{remark}
Let $p = (v_1, \dots, v_{i+1})$ and $q = (w_1, \dots, w_{i+1})$ be distinct elements of $P_i$. The ordering of $p$ and $q$ is determined by the choice of ordering of $(v_1, \dots, v_j)$ and $(w_1, \dots, w_j)$, where $j$ is minimal such that $v_j \neq w_j$. This in turn may be constructed from a choice of ordering of $(v_{j-1}, v_j)$ and $(w_{j-1}, w_j) = (v_{j-1}, w_j)$. We conclude the following: If for all vertices $w$ of $T$ we know a well-ordering of the edges incident to $w$, then we can determine $<_i$. This provides a method to implement the well-orderings.
\end{remark}
\begin{definition}\label{def:more-orderings} Let $\nset{P} = \Set{A \subseteq P \given \abs{A} = 2}$.
  \begin{itemize}
    \item
    Define a well-ordering $\prec$ on $\nset{P}$ such that $A \prec B$ if and only if
    \[
      \min\limits_{<}\left((A \cup B) \setminus (A \cap B)\right) \in A.
    \]
    
    \item
    Define a partial order $\prec$ on $G$ such that $g \prec h$ if and only if $\abs{g} = \abs{h}$ and $\{\half(g), \half(g^{-1})\} \prec \{\half(h), \half(h^{-1})\}$.
    
    \item
    Define a pre-order $\leqg$ on $G$ such that $g \leqg h$ if and only if $\{\half(g), \half(g^{-1})\} \preceq \{\half(h), \half(h^{-1})\}$.
  \end{itemize}
\end{definition}
Note that $g \ltg h$ if and only if either $\abs{g} < \abs{h}$ or $g \prec h$. Also, $g$ and $h$ are in the same connected component of $\leqg$ if and only if $\{g\vo, g^{-1}\vo\} = \{h\vo, h^{-1}\vo\}$. Hence the induced ordering on the connected components of $\leqg$ is isomorphic to the ordering $\prec$ on $\nset{P}$. Therefore $\leqg$ is a pre-well-ordering of $G$.
\begin{lemma}\label{lem:N3-violate-overlap}
  Suppose $X$ satisfies \textup{N2}. If there exists $x, y, z \in X^\pm$ such that $y \notin \{x^{-1}, z^{-1}\}$ and $\abs{xyz} \leq \abs{x} + \abs{z} - \abs{y}$, then $\delta(x^{-1}, y) = \delta(y^{-1}, z) = \abs{y}/2$.
\end{lemma}
\begin{proof}
  Since $X$ violates N3, it violates N3$'$; in particular, $\delta(x^{-1}, y) + \delta(y^{-1}, z) \geq \abs{y}$. In addition, by N2$'$, $\max\{\delta(x^{-1}, y), \delta(y^{-1}, z)\} \leq \abs{y}/2$.
  Therefore $\delta(x^{-1}, y) = \delta(y^{-1}, z) = \abs{y}/2$.
\end{proof}
Algorithm \ref{alg:reduce} takes as input $X$, and outputs a basis $X'$ of $G$ that is either \reduced or contains a non-trivial elliptic element; it thereby establishes Theorem \ref{thm:main}.
\begin{algorithm}[ht]
  \DontPrintSemicolon
  \KwData{Finite subset $X$ of $\Aut(T)$}
  \KwOut{$(\mathit{flag}, X')$ where $\gen{X'} = \gen{X} = G$, $\mathit{flag} = \tt{True}$ if $G$ is purely hyperbolic and $X'$ is a free basis of $G$, and $\mathit{flag} = \tt{False}$ if $X'$ contains an elliptic element}
  $X' \gets X$\;
  \Loop{
    \uIf{there exists non-trivial elliptic $x \in X'$} {
      \Return $(\texttt{False}, X')$\;
    }
    \uElseIf{there exists $x \in X' \cap X'^{-1}$}{
      $X' \gets X' \setminus \{x\}$\;
    } \uElseIf{there exists $x, y \in X'^\pm$ such that $x \neq y^{-1}$ and $xy \ltg x$}{
      $X' \gets X' \setminus \{x, x^{-1}\} \cup \{xy\}$\;
    } \Else{
      \Return $(\texttt{True}, X')$ \;
    }
  }
  \caption{Decide whether a finitely generated subgroup of $\Aut(T)$ is purely hyperbolic}
  \label{alg:reduce}
\end{algorithm}
\begin{figure}[t]
  \centering
  \begin{tikzpicture}
	\begin{pgfonlayer}{nodelayer}
		\node [style=vertex, label={above:$x\vo$}] (0) at (-3, 0) {};
		\node [style=vertex, label={above:$xy\vo$}] (1) at (-6, 2) {};
		\node [style=vertex, label={above:$\vo$}] (4) at (-11, 0) {};
		\node [style=none] (16) at (-6, 0) {};
		\node [style=vertex, label={above:$x'y^{-1}\vo$}] (25) at (11, 0) {};
		\node [style=vertex, label={above:$x'\vo$}] (26) at (8, 2) {};
		\node [style=vertex, label={above:$\vo$}] (27) at (3, 0) {};
		\node [style=none] (28) at (8, 0) {};
		\node [style=none] (30) at (-1, 0) {};
		\node [style=none] (31) at (1, 0) {};
	\end{pgfonlayer}
	\begin{pgfonlayer}{edgelayer}
		\draw [style=edge, color=NiceRed] (16.center) to (4);
		\draw [style=edge, color=NiceRed] (16.center) to (0);
		\draw [style=edge] (16.center) to (1);
		\draw [style=edge, color=NiceRed] (28.center) to (27);
		\draw [style=edge] (28.center) to (25);
		\draw [style=edge, color=NiceRed] (28.center) to (26);
		\draw [style=implication] (30.center) to (31.center);
	\end{pgfonlayer}
\end{tikzpicture}
  \caption{A type 2 transformation $x \mapsto x'$}
  \label{fig:a2swap}
\end{figure}
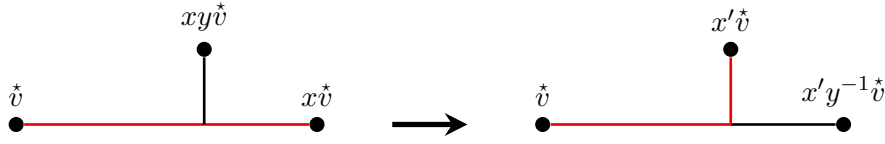
\begin{proof}[Proof of correctness of Algorithm \emph{\ref{alg:reduce}}]
  Let $X_i$ be the value of $X'$ during the $i$-th iteration of the loop. If $X_i$ has a non-trivial elliptic element, then we terminate. Suppose $X_i$ has no non-trivial elliptic elements.
  \begin{enumerate}
    \item
    If $X_i$ violates N1, then there is some $x \in X_i \cap X_i^{-1}$. In this case, we assign $X_{i+1} = X_i \setminus \{x\}$.
    \item
    If $X_i$ satisfies N1 and violates N2, then there is some $x, y \in X_i^{\pm}$ with $x \neq y^{-1}$ such that $\abs{xy} < \abs{x}$ (as in Figure \ref{fig:a2swap}). Since $x$ is hyperbolic, by Proposition \ref{prop:translation-length-formula} $x \neq y$, so we replace $x$ by $xy$.
    \item
    Suppose $X_i$ satisfies N1 and N2, and violates N3. Let $x, y, z \in X_i^{\pm}$ be such that $y \notin \{x^{-1}, z^{-1}\}$ and $\abs{xyz} \leq \abs{x} + \abs{z} - \abs{y}$. Define
    \[
    p = [\vo, y\vo] \cap [\vo, x^{-1}\vo],
    \qquad q = [\vo, y^{-1}\vo] \cap [\vo, z\vo],
    \]
    so that $\abs{p} = \delta(x^{-1}, y)$ and $\abs{q} = \delta(y^{-1}, z)$, as in Figure \ref{fig:a3pq}. Note that the right diagram is isometric to the left diagram via multiplication by $y^{-1}$, and contains an isometric copy of Figure \ref{fig:half-path}. By Lemma \ref{lem:N3-violate-overlap} and N2$'$,
    \[
    \abs{p}=\abs{q} = \abs{y}/2 \leq \min\left\{\floor*{\frac{\abs{x}}{2}}, \floor*{\frac{\abs{z}}{2}}\right\}.
    \]
    From Figure \ref{fig:a3pq} we see that $\abs{xy} = \abs{x}$ and $\abs{yz} = \abs{z}$.
    Hence $xp = [\vo, x\vo] \cap [x\vo, xy\vo]$ and $xyq = [x\vo, xy\vo] \cap [xy\vo, xyz\vo]$ intersect at a unique point $w$.
    
    Suppose $y=x$. Then $\abs{p} = \abs{x}/2$. By Proposition \ref{prop:translation-length-formula}, $\abs{x^2} = 2\abs{x} - 2\abs{p} = \abs{x}$, hence $x$ is elliptic. Similarly, if $y=z$, then $z$ is elliptic. If $p = q$, then $x^{-1}w = (xy)^{-1}w$, so $y$ is elliptic. Thus we may assume that $y \notin \{x, z\}$ and $p \neq q$.
    
    By Lemma \ref{lem:same-initial-segment}, $\half(x) = \half(xy)$ and $\half(z^{-1}) = \half((yz)^{-1})$. Note also that $p$ is an initial segment of both $\half(x^{-1})$ and $\half(yz)$, and $q$ is an initial segment of both $\half((xy)^{-1})$ and $\half(z)$. If $q < p$, then $\half((xy)^{-1}) < \half(x^{-1})$ so $xy \prec x$, and we replace $x$ by $xy$, as in Figure \ref{fig:a3swap}. Similarly, if $p < q$, then $\half(yz) < \half(z)$ and $yz \prec z$, in which case we replace $z$ by $yz$.
  \end{enumerate}
  A transformation performed in case 1, 2 or 3 is denoted \emph{type} 1, 2 or 3 respectively. Case 1 reduces $\abs{X_i}$, so only finitely many transformations are type 1. Cases 2 and 3 reduce some element of $X_i$ with respect to $\ltg$, and does not increase $\abs{X_i}$. Therefore each generator may be replaced only finitely many times, so finitely many transformations can be types 2 and 3. By Proposition \ref{prop:reduced-implies-free}, we either find a non-trivial elliptic element of $G$, or a generating set $X'$ of $G$ that is \reduced and therefore a free basis for $G$.
\end{proof}
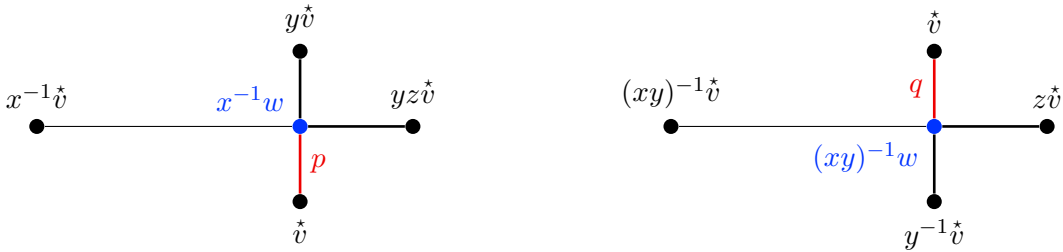
\begin{figure}[ht]
  \centering
  \begin{subfigure}{0.5\textwidth}
    \centering
    \begin{tikzpicture}
	\begin{pgfonlayer}{nodelayer}
		\node [style=vertex, label={below:$\vo$}] (0) at (0, -2) {};
		\node [style=vertex, label={above:$y\vo$}] (1) at (0, 2) {};
		\node [style=vertex, label={above:$x^{-1}\vo$}] (4) at (-7, 0) {};
		\node [style=vertex, color=NiceBlue, label={[text=NiceBlue]above left:$x^{-1}w$}] (16) at (0, 0) {};
		\node [style=vertex, label={above:$yz\vo$}] (24) at (3, 0) {};
	\end{pgfonlayer}
	\begin{pgfonlayer}{edgelayer}
		\draw [style=edge, color=NiceRed] (16) to node [right] {$p$} (0);
		\draw [style=edge] (16) to (1);
		\draw [style=edge] (16) to (24);
		\draw (4) to (16);
	\end{pgfonlayer}
\end{tikzpicture}
  \end{subfigure}%
  \begin{subfigure}{0.5\textwidth}
    \centering
    \begin{tikzpicture}
	\begin{pgfonlayer}{nodelayer}
		\node [style=vertex, label={below:$y^{-1}\vo$}] (0) at (0, -2) {};
		\node [style=vertex, label={above:$\vo$}] (1) at (0, 2) {};
		\node [style=vertex, label={above:$(xy)^{-1}\vo$}] (4) at (-7, 0) {};
		\node [style=vertex, color=NiceBlue, label={[text=NiceBlue]below left:$(xy)^{-1}w$}] (16) at (0, 0) {};
		\node [style=vertex, label={above:$z\vo$}] (24) at (3, 0) {};
	\end{pgfonlayer}
	\begin{pgfonlayer}{edgelayer}
		\draw [style=edge] (16) to (0);
		\draw [style=edge, color=NiceRed] (1) to node [left] {$q$} (16);
		\draw [style=edge] (16) to (24);
		\draw (4) to (16);
	\end{pgfonlayer}
\end{tikzpicture}
  \end{subfigure}
  \caption{The paths $p$ and $q$}
  \label{fig:a3pq}
\end{figure}
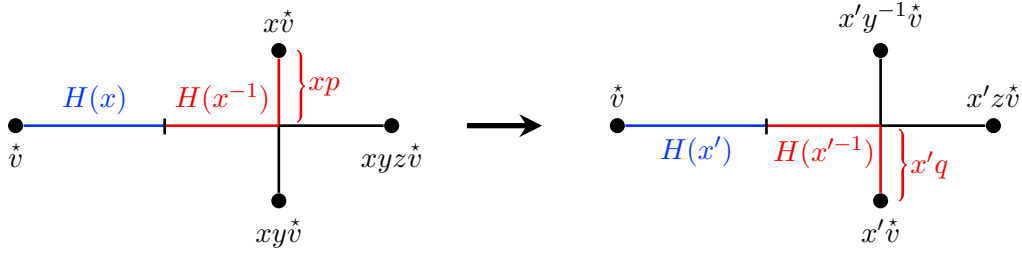
\begin{figure}[ht]
  \centering
  \begin{tikzpicture}
	\begin{pgfonlayer}{nodelayer}
		\node [style=vertex, label={above:$x\vo$}] (0) at (-6, 2) {};
		\node [style=vertex, label={below:$xy\vo$}] (1) at (-6, -2) {};
		\node [style=vertex, label={below:$\vo$}] (4) at (-13, 0) {};
		\node [style=none] (16) at (-6, 0) {};
		\node [style=vertex, label={below:$xyz\vo$}] (24) at (-3, 0) {};
		\node [style=vertex, label={above:$x'y^{-1}\vo$}] (25) at (10, 2) {};
		\node [style=vertex, label={below:$x'\vo$}] (26) at (10, -2) {};
		\node [style=vertex, label={above:$\vo$}] (27) at (3, 0) {};
		\node [style=none] (28) at (10, 0) {};
		\node [style=vertex, label={above:$x'z\vo$}] (29) at (13, 0) {};
		\node [style=none] (30) at (-1, 0) {};
		\node [style=none] (31) at (1, 0) {};
		\node [style=none] (36) at (-5.5, 2) {};
		\node [style=none] (37) at (-5.5, 0.1) {};
		\node [style=none] (38) at (10.5, -0.1) {};
		\node [style=none] (39) at (10.5, -2) {};
		\node [style=vertex-dash-vertical] (40) at (-9, 0) {};
		\node [style=vertex-dash-vertical] (41) at (7, 0) {};
	\end{pgfonlayer}
	\begin{pgfonlayer}{edgelayer}
		\draw [style=edge, color=NiceRed] (16.center) to (0);
		\draw [style=edge] (16.center) to (1);
		\draw [style=edge] (16.center) to (24);
		\draw [style=edge] (28.center) to (25);
		\draw [style=edge, color=NiceRed] (28.center) to (26);
		\draw [style=edge] (28.center) to (29);
		\draw [style=implication] (30.center) to (31.center);
		\draw [style=brace, pen colour=NiceRed, color=NiceRed] (36.center) to node [auto] {$xp$} (37.center);
		\draw [style=brace, pen colour=NiceRed, color=NiceRed] (38.center) to node [auto] {$x'q$} (39.center);
		\draw [style=edge, color=NiceRed] (16.center) to node [above] {$\half(x^{-1})$} (40.center);
		\draw [style=edge, color=NiceBlue] (40.center) to node [above] {$\half(x)$} (4);
		\draw [style=edge, color=NiceRed] (28.center) to node [below] {$\half(x'^{-1})$} (41.center);
		\draw [style=edge, color=NiceBlue] (41.center) to node [below] {$\half(x')$} (27);
	\end{pgfonlayer}
\end{tikzpicture}
  \caption{A type 3 transformation $x \mapsto x'$}
  \label{fig:a3swap}
\end{figure}
To prove Theorem \ref{thm:mainalgo} we use a generalisation of Ihara's theorem \cite[Chapter II.1.5, Theorem 4]{serre-trees}. This proves \cite[Theorem 2.5]{conder} independently of \cite[Conjecture 2.3]{conder}.
\begin{proposition}\label{prop:ihara-generalisation}
  Let $G$ be a subgroup of $\Aut(T)$.
  \begin{enumerate}
  \item
  If $T$ is locally finite and $G$ is discrete and free, then $G$ is purely hyperbolic.
  \item
  If $G$ is purely hyperbolic, then $G$ is discrete and free.
  \end{enumerate}
\end{proposition}
\begin{proof}
  (1) is proved in \cite[Proposition 5.1]{conder-2020}. To prove (2): Freeness is given by the Bass-Serre Theorem \cite[Chapter I.3.3, Theorem 4]{serre-trees}. Let $(g_i)_{i \in \N}$ be a sequence of elements of $G$ converging to $1$. Since the stabiliser of a vertex is an open neighbourhood of $1$, all but finitely many $g_i$ are trivial, hence $G$ is discrete.
\end{proof}
\begin{remark}
  Here is a counterexample to the converse of Proposition \ref{prop:ihara-generalisation} (2), showing the need for local finiteness in (1): Suppose $T$ is regular and $\vo$ has neighbourhood set $\Set{v_i \given i \in \Z}$. Let $g \in \Aut(T)$ be such that $g\vo = \vo$ and $gv_i = v_{i+1}$ for all $i \in \Z$. Now $g$ is elliptic, but $\gen{g}$ is discrete and free.
\end{remark}
\begin{proof}[Proof of Theorem \ref{thm:mainalgo}]
  Suppose $T$ is locally finite. Then $G$ is discrete and free if and only if $G$ is purely hyperbolic, so Algorithm \ref{alg:reduce} decides whether $G$ is discrete and free.
\end{proof}

\section{Strong \reduction}\label{sec:strongly-reduce}
For the rest of the paper we fix a choice of well-orderings $(<_i)_{i \in \N}$, and therefore the orderings in Definitions \ref{def:well-orderings} and \ref{def:more-orderings}.
\begin{definition}\label{def:strongly-reduced}
  A free basis $X$ for $G$ is \emph{strongly \reduced} (with respect to $\vo$ and $\leqg$) if $X$ satisfies N1 and:
  \begin{enumerate}
    \item[N4.] For all $x, y \in X^{\pm}$, if $x \neq y^{-1}$, then $x \ltg xy$.
  \end{enumerate}
\end{definition}
We could equivalently write $y \ltg xy$, since $y \ltg xy$ if and only if $y^{-1} \ltg y^{-1}x^{-1}$.

Note that a free basis returned by Algorithm \ref{alg:reduce} is strongly \reduced. In particular, if $X$ is strongly \reduced with respect to $\vo$ and $\leqg$, then it is \reduced with respect to $\vo$, since Algorithm \ref{alg:reduce} finds an \reduced basis.

The following lemmas show that a word in a strongly \reduced set is bounded from below (with respect to $\leqg$) by its subwords.
\begin{lemma}\label{lem:word-size-bound}
  Suppose $X$ is strongly \reduced. Let $g = a_1 \dots a_n$ be a reduced word in $X$. For all $1 \leq i \leq j \leq n$, $\abs{g} \geq \abs{a_i \dots a_j}$.
\end{lemma}
\begin{proof}
  The case $n=1$ is trivial. We may assume $i = 1$ and $j = n-1$: Applying this case to $g^{-1} = a_n^{-1}\dots a_1^{-1}$ proves the case $i=2$ and $j=n$, from which the rest follows by transitivity.
  
  Define $h = a_1\dots a_{n-1}$. By Lemma \ref{lem:no-overlap} (2) and N2$'$,
  \[
    \abs{g} = \abs{h} + \abs{a_n} - 2\delta(h^{-1}, a_n) = \abs{h} + \abs{a_n} - 2\delta(a_{n-1}^{-1}, a_n) \geq \abs{h}. \qedhere
  \]
\end{proof}
\begin{lemma}\label{lem:half-subset}
  Suppose $X$ is strongly \reduced. Let $g = a_1 \dots a_n$ be a reduced word in $X$. For all $1 \leq i < n$, $\half(a_1 \dots a_i) \subseteq \half(g)$.
\end{lemma}
\begin{proof}
  By transitivity, we assume $i = n-1$.
  Let $h = a_1 \dots a_{n-1}$. By Lemma \ref{lem:no-overlap} (2),
  \[
    \abs{[\vo, h \vo] \cap [\vo, g\vo]} = \abs{h} - \delta(h^{-1}, a_{n})
    = \abs{h} - \delta(a_{n-1}^{-1}, a_{n}).
  \]
  By N2$'$ and Lemma \ref{lem:word-size-bound},
  \[
    \abs{h} - \delta(a_{n-1}^{-1}, a_{n}) \geq \abs{h} - \frac{\abs{a_{n-1}}}{2}
    \geq \frac{\abs{h}}{2}.
  \]
  Thus $\half(h) \subseteq [\vo, h \vo] \cap [\vo, g\vo] \subseteq [\vo, g\vo]$. Since $\abs{\half(h)} \leq \abs{\half(g)}$, we conclude that $\half(h) \subseteq \half(g)$. 
\end{proof}
\begin{lemma}\label{lem:word-ordering-increases}
  Suppose $X$ is strongly \reduced. Let $g = a_1 \dots a_n$ be a reduced word in $X$. Let $h = a_1 \dots a_i$ for some $1 \leq i < n$. Suppose $\abs{h} = \abs{g} = s$.
  \begin{enumerate}
    \item For all $i \leq j \leq n$, $\abs{a_1 \dots a_{j}} = s$.
    \item $\abs{a_n} \leq \abs{a_{n-1}}$, with strict inequality when $i < n-1$.
    \item For all $i < j \leq n$, $h \prec a_1 \dots a_j \preceq g$.
  \end{enumerate}
\end{lemma}
\begin{proof}
  (1) follows directly from Lemma \ref{lem:word-size-bound}. We now prove (2). By Lemma \ref{lem:no-overlap} (3),
  \[
    0 = \abs{a_1 \dots a_{n}} - \abs{a_1 \dots a_{n-1}} = \abs{a_n} - 2\delta(a_{n-1}^{-1}, a_n),
  \]
  so $\delta(a_{n-1}^{-1}, a_n) = \abs{a_n}/2$. By N2$'$, $\abs{a_n} \leq \abs{a_{n-1}}$.
  If $i < n-1$ then, by similar reasoning, \[
    \delta(a_{n-2}^{-1}, a_{n-1}) = \abs{a_{n-1}}/2 \geq \delta(a_{n-1}^{-1}, a_n).
  \] By N3$'$,
  \[
    \abs{a_n} = 2\delta(a_{n-1}^{-1}, a_{n})\leq \delta(a_{n-2}^{-1}, a_{n-1}) + \delta(a_{n-1}^{-1}, a_{n}) < \abs{a_{n-1}}.
  \]
  Lastly we prove (3). By transitivity, we may assume $i = n-1$ and $j=n$. Let $w \in V(T)$ be such that $[\vo, w] = [\vo, h\vo] \cap [\vo, g\vo]$, as in Figure \ref{fig:ordering-increases}. By Lemma \ref{lem:no-overlap} (2) and N2$'$, 
  \[
    d(\vo, w) = \abs{h} - \delta(h^{-1}, a_n) = \abs{h} - \delta(a_{n-1}^{-1}, a_n) \geq \frac{\abs{h}}{2},
  \]
  so $\half(h) = \half(g)$. We now show that $\half(h^{-1}) < \half(g^{-1})$. By Lemma \ref{lem:half-subset}, $\half(a_{n-1}^{-1}) \subseteq \half(h^{-1})$ and $\half(a_n^{-1}) \subseteq \half(g^{-1})$. Additionally, by Definition \ref{def:delta},
  \[
    \abs{a_n a_{n-1}} = \abs{a_n} + \abs{a_{n-1}} - 2\delta(a_{n-1}^{-1}, a_n) = \abs{a_{n-1}}.
  \]
  By Lemma \ref{lem:same-initial-segment}, $\half(a_{n-1}) = \half(a_{n-1}a_n)$. By N4, $a_{n-1} \prec a_{n-1}a_n$. Thus $\half(a_{n-1}^{-1}) \ltp \half((a_{n-1}a_n)^{-1})$. By Lemma \ref{lem:half-subset}, $\half((a_{n-1}a_n)^{-1}) \subseteq \half(g^{-1})$, hence $\half(h^{-1}) < \half(g^{-1})$. Therefore $h \prec g$.
\end{proof}
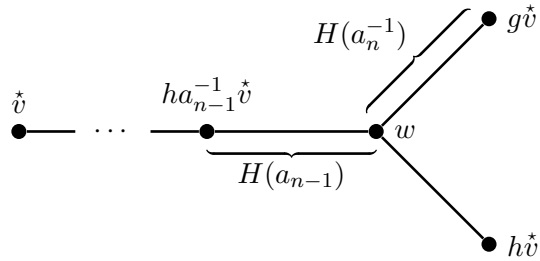
\begin{figure}[ht]
  \centering
  \begin{tikzpicture}
	\begin{pgfonlayer}{nodelayer}
		\node [style=vertex, label={right:$w$}] (0) at (4.5, 0) {};
		\node [style=vertex, label={above:$\vo$}] (1) at (-5, 0) {};
		\node [style=vertex, label={right:$h\vo$}] (2) at (7.5, -3) {};
		\node [style=vertex, label={right:$g\vo$}] (3) at (7.5, 3) {};
		\node [style=vertex, label={above:$ha_{n-1}^{-1}\vo$}] (4) at (0, 0) {};
		\node [style=none] (5) at (4.25, 0.5) {};
		\node [style=none] (6) at (7, 3.25) {};
		\node [style=none] (7) at (-3.5, 0) {};
		\node [style=none] (8) at (-1.5, 0) {};
		\node [style=none] (9) at (-2.5, 0) {$\cdots$};
		\node [style=none] (10) at (0, -0.5) {};
		\node [style=none] (11) at (4.5, -0.5) {};
	\end{pgfonlayer}
	\begin{pgfonlayer}{edgelayer}
		\draw [style=edge] (4) to (0);
		\draw [style=edge] (0) to (2);
		\draw [style=edge] (0) to (3);
		\draw [style=brace] (5.center) to node [auto] {$\half(a_n^{-1})$} (6.center);
		\draw [style=edge] (1) to (7.center);
		\draw [style=edge] (8.center) to (4);
		\draw [style=brace] (11.center) to node [auto] {$\half(a_{n-1})$} (10.center);
	\end{pgfonlayer}
\end{tikzpicture}
  \caption{The construction for Lemma \ref{lem:word-ordering-increases}}\label{fig:ordering-increases}
\end{figure}
\begin{theorem}\label{thm:reduced-is-unique}
  Let $X$ and $Y$ be finite generating sets of a purely hyperbolic subgroup of $\Aut(T)$. If $X$ and $Y$ are strongly \reduced (with respect to $\vo$ and $\ltg$), then $X^\pm = Y^\pm$.
\end{theorem}
\begin{proof}
  We prove this by induction on $n = \max\Set{\abs{g} \given g \in X \cup Y}$. Define \[
    X_{<n} = \Set{g \in X \given \abs{g} < n},\qquad X_{n} = \Set{g \in X \given \abs{g} = n},
  \]
  and define $Y_{<n}$ and $Y_n$ similarly. By Lemma \ref{lem:word-size-bound}, given some $y \in Y_{<n}$, the reduced word for $y$ in $X$ cannot contain an element of $X_n^\pm$, and hence is a word in $X_{<n}$. Similarly, every $x \in X_{<n}$ can be written as a reduced word in $Y_{<n}$. It follows that $\gen{X_{<n}} = \gen{Y_{<n}}$. Note also that $X_{<n}$ and $Y_{<n}$ are strongly \reduced. If $n = 1$, then $X_{<n} = Y_{<n} = \varnothing$. Otherwise, by the induction hypothesis $X_{<n}^\pm = Y_{<n}^\pm$.
  
  Let $x \in X_n$ have reduced word $y_1 \dots y_m$ in $Y$. Since $x$ is not represented by a word in $Y_{< n}$, there must be some $y_j \in Y^{\pm}_n$. Suppose for a contradiction that $y_k \in X_{< n}^\pm$ for some $k > j$. By Lemma \ref{lem:word-size-bound}, $\abs{y_j \dots y_k \dots y_m} = n$, hence it follows from Lemma \ref{lem:word-ordering-increases} (2) that $y_m \in X_{<n}^\pm$. But Lemma \ref{lem:word-ordering-increases} (3) implies that $x y_m^{-1} \prec x$, contradicting N4. Applying the same argument to $x^{-1}$, we find that $y_k \in Y_n$ for all $1 \leq k \leq m$. By applying Lemma \ref{lem:word-ordering-increases} (2) again to $x$, we conclude that $m \leq 2$.
  
  Suppose $m=2$. Write $x = y y'$, where $y, y' \in Y_{n}^\pm$. Since the words of $y$ and $y'$ in $X_n$ have length at most 2, and $y y'$ reduces to an element of $X_n^\pm$, we may assume (up to inverting $x$) that $y = x'$ and $y' = x'^{-1} x$ for some $x' \in X_n^\pm$. By N4, $y y' = x \prec y'$, which is a contradiction. Therefore $m = 1$ and $x \in Y^\pm$. By symmetry, $X^\pm = Y^\pm$.
\end{proof}
As a consequence, we can decide equality of purely hyperbolic subgroups of $\Aut(T)$.
\begin{corollary}\label{cor:equal-subgroup}
  Let $X$ and $Y$ be finite subsets of $\Aut(T)$ generating purely hyperbolic subgroups $G$ and $H$ respectively. There exists an algorithm to decide whether $G = H$.
\end{corollary}
\begin{proof}
  We use Algorithm \ref{alg:reduce} to find strongly \reduced free bases $X'$ and $Y'$ for $G$ and $H$ respectively. By Theorem \ref{thm:reduced-is-unique}, $G = H$ if and only if $X'^\pm = Y'^\pm$.
\end{proof}

\section{The constructive membership problem}\label{sec:constructive-membership}
In this section we present a solution to the constructive membership problem for finitely generated purely hyperbolic subgroups of $\Aut(T)$, by taking advantage of the properties of strongly \reduced bases.

Recall that we have fixed a tree $T$, a vertex $\vo$, a lexicographic well-ordering $\ltp$ on the set of paths on $T$ starting from $\vo$ as in Definition \ref{def:well-orderings}, and a finite subset $X$ of $\Aut(T)$ generating a group $G$.
\begin{definition}
  Let $g\in G$ be hyperbolic. The \emph{translation axis} of $g$ is $\Ax(g) = \Min(g)$.
\end{definition}
Since $g$ acts on $\Ax(g)$ by translation, there is a natural direction associated with the action of $g$ on $\Ax(g)$ such that $g$ and $g^{-1}$ translate in opposite directions. Let $x$ and $y$ be distinct elements of $\Ax(g)$. If $gy$ and $x$ lie in the same connected component of $\Ax(g) \setminus \{y\}$, then $g$ translates $y$ \emph{towards} $x$. Otherwise, $g$ translates $y$ \emph{away from} $x$.

The Ping-Pong Lemma is a well-known tool for proving the freeness of a group. We use the following incarnation \cite[Lemma 3.1]{conder}:
\begin{lemma}[The Ping Pong Lemma]
  Let $X$ be a finite set of hyperbolic automorphisms of $T$. Suppose that for each $g \in X$ there is an open segment $P_g \subseteq \Ax(g)$ of length $l(g)$ such that
  \[
    \bigcup_{\mathclap{h \in X \setminus \{g\}}}\; \proj_{\Ax(g)}(\Ax(h)) \subseteq P_g.
  \]
  Then the group $G$ generated by $X$ is free and $X$ is a free basis for $G$.
\end{lemma}
Given some hyperbolic $g \in \Aut(T)$, let $U^0_g$ be the half-open interval $(u^-_g, u^+_g]$ of $\Ax(g)$ with radius $l(g)/2$ centred at $\proj_{\Ax(g)}(\vo)$, such that $[\vo, u^+_g] < [\vo, u^-_g]$. Let $U^-_g$ and $U^+_g$ be the components of $\Ax(g) \setminus U^0_g$ such that $u^-_g \in U^-_g$. Define $U_g^\pm = U_g^+ \cup U_g^-$. Define $U^0_{g^{-1}} = U^0_g$, $u^+_{g^{-1}} = u^+_g$, and other values similarly. Note that $U_g^+$, $U_g^-$, $u_g^+$, and $u_g^-$ are independent of the direction along which $g$ translates $\Ax(g)$.
\begin{lemma}\label{lem:translate-half}
  Let $g \in \Aut(T)$ be hyperbolic.
  \begin{enumerate}
    \item There is a unique $u_g \in \{u_g^+, u_g^-\} \cap [\vo, g\vo]$. Furthermore, $\abs{g} = 2d(\vo, u_g)$.
    \item $[\vo, g\vo] \cap U_g^{\pm} \neq \varnothing$.
  \end{enumerate}
\end{lemma}
\begin{proof}
  Let $a = \proj_{\Ax(g)}(\vo)$, as in Figure \ref{fig:axis-partition}. We may write $[\vo, g\vo] = [\vo, a] \cup [a, ga] \cup [ga, g\vo]$.
  Let $u_g$ be the midpoint of $[a, ga]$. Now $d(a, u_g) = d(u_g, a) = l(g)/2$, and $u_g$ is the unique point of $[a, ga]$ with this property. Hence $\{u_g^+, u_g^-\} \cap [\vo, g\vo] = \{u_g\}$. Also,
  \[
    \abs{g} = 2d(\vo, a) + 2d(a, u_g) = 2d(\vo, u_g),
  \] proving (1). Since $u_g$ lies in the boundary of $U_g^\pm$, it follows that $ga \in U_g^\pm$, proving (2).
\end{proof}
\begin{figure}[ht]
  \centering
  \begin{tikzpicture}
	\begin{pgfonlayer}{nodelayer}
		\node [style=vertex, label={above right:$a$}] (0) at (0, 0) {};
		\node [style=none, label={right:$\Ax(g)$}] (1) at (9, 0) {};
		\node [style=none, label={left:$\Ax(g)$}] (2) at (-6, 0) {};
		\node [style=none, label={above:$u_{g}^+$}] (3) at (3, 0) {]};
		\node [style=none, label={above:$u_{g}^-$}] (4) at (-3, 0) {(};
		\node [style=vertex, label={above:$\vo$}] (7) at (0, 4) {};
		\node [style=vertex, label={above right:$ga$}] (8) at (6, 0) {};
		\node [style=vertex, label={above:$g\vo$}] (9) at (6, 4) {};
	\end{pgfonlayer}
	\begin{pgfonlayer}{edgelayer}
		\draw [style=edge-dash] (7) to (0);
		\draw [style=edge] (2.center) to node [midway, label={below:$U_g^-$}] {} (4.center);
		\draw [style=edge] (3.center) to node [pos=1, midway, label={below:$U_g^+$}] {} (8);
    \draw [style=edge] (8) to (1.center);
		\draw [style=edge] (4.center) to node [pos=1, label={below:$U_g^0$}] {} (0);
    \draw [style=edge] (0) to (3.center);
		\draw [style=edge-dash] (9) to (8);
	\end{pgfonlayer}
\end{tikzpicture}
  \caption{The line segments  $U_g^-$, $U_g^0$, and $U_g^+$}
  \label{fig:axis-partition}
\end{figure}
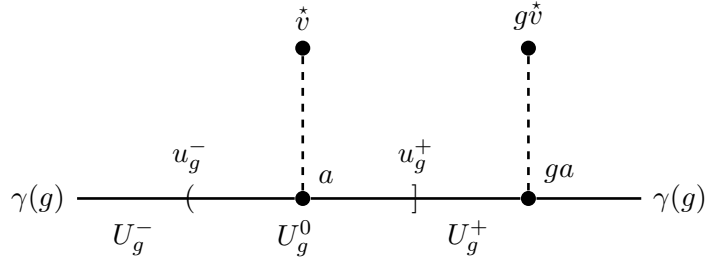
\begin{definition}\label{def:fundamental-system}
  Define $U = \bigcup_{g \in X} \{U^-_g, U^0_g, U^+_g\}$.
  If $X$ has no elliptic elements and $\proj_{\Ax(g)}(\Ax(h)) \subseteq U^0_g$ for all $h \in X \setminus \{g\}$, then $U$ is the \emph{fundamental system} of $X$.
\end{definition}
It is clear that if $X$ admits a fundamental system, then $X$ satisfies the hypotheses of the Ping-Pong Lemma, and is therefore a free basis of $G$.

The following lemmas show how paths on $T$ interact with the translation axes of hyperbolic elements of $\Aut(T)$. 
\begin{lemma}\label{lem:translate-overlap}
  Let $w \in T$ and let $g \in \Aut(T)$ be hyperbolic. Define
  \[
    x = \proj_{\Ax(g)}(\vo), \quad y = \proj_{\Ax(g)}(w), \quad p = [\vo, w] \cap \Ax(g), \quad q = [\vo, gw] \cap \Ax(g).
  \]
  \begin{enumerate}
    \item If $p = \varnothing$, then $q = [x, gx]$ and $d(\vo, gw) > d(\vo, w) + l(g)$.
    
    \item If $p = [x, y]$ and $gy = x$, then $q \in \{\varnothing, \{x\}\}$ and $d(\vo, gw) \leq d(\vo, w) - l(g)$.
    
    \item If $p = [x, y]$ and either $x = y$ or $g$ translates $y$ away from $x$, then $q = [x, gy]$ and $d(\vo, gw) = d(\vo, w) + l(g)$.
    
    \item If $p = [x, y]$ where $x \notin \{y, gy\}$ and $g$ translates $y$ towards $x$, then $q = [x, gy]$ and $d(\vo, gw) = d(\vo, w) + \abs[\big]{d(x, y) - l(g)} - d(x, y)$.
  \end{enumerate}
\end{lemma}
\begin{midfigure}[ht]
  \renewcommand\thesubfigure{\arabic{subfigure}}
  \centering
  \begin{subfigure}[b]{0.5\textwidth}
    \centering
    \begin{tikzpicture}
	\begin{pgfonlayer}{nodelayer}
		\node [style=vertex, label={below:$x=y$}] (0) at (-1, 0) {};
		\node [style=vertex, label={above:$\vo$}] (1) at (-1, 5) {};
		\node [style=none, label={right:$\Ax(g)$}] (2) at (6, 0) {};
		\node [style=none, label={left:$\Ax(g)$}] (3) at (-3, 0) {};
		\node [style=none] (4) at (-1, 3) {};
		\node [style=vertex, label={above:$w$}] (5) at (1, 5) {};
		\node [style=vertex, label={above:$gw$}] (6) at (4, 5) {};
		\node [style=vertex, label={below:$gx$}] (7) at (4, 0) {};
		\node [style=none] (8) at (-1, -1) {};
		\node [style=none] (9) at (4, -1) {};
	\end{pgfonlayer}
	\begin{pgfonlayer}{edgelayer}
		\draw [style=edge, >-] (3.center) to (0);
		\draw [style=edge] (1) to (4.center);
		\draw [style=edge] (4.center) to (5);
		\draw [style=edge] (4.center) to (0);
		\draw [style=edge] (6) to (7);
		\draw [style=edge] (0) to (7);
		\draw [style=edge, ->] (7) to (2.center);
		\draw [style=brace, pen colour=NiceBlue, color=NiceBlue] (9.center) to node [auto] {$q$} (8.center);
	\end{pgfonlayer}
\end{tikzpicture}
    \caption{$p = \varnothing$}
    \label{fig:translate1}
  \end{subfigure}%
  \hfill
  \begin{subfigure}[b]{0.5\textwidth}
    \centering
    \begin{tikzpicture}
	\begin{pgfonlayer}{nodelayer}
		\node [style=vertex, label={below:$x=gy$}] (0) at (4, 0) {};
		\node [style=vertex, label={above:$\vo$}] (1) at (4, 5) {};
		\node [style=none, label={right:$\Ax(g)$}] (2) at (6, 0) {};
		\node [style=none, label={left:$\Ax(g)$}] (3) at (-3, 0) {};
		\node [style=none] (4) at (4, 3) {};
		\node [style=vertex, label={above:$gw$}] (5) at (6, 5) {};
		\node [style=vertex, label={above:$w$}] (6) at (-1, 5) {};
		\node [style=vertex, label={below:$y$}] (7) at (-1, 0) {};
		\node [style=none] (8) at (-1, -1) {};
		\node [style=none] (9) at (4, -1) {};
	\end{pgfonlayer}
	\begin{pgfonlayer}{edgelayer}
		\draw [style=edge] (1) to (4.center);
		\draw [style=edge] (4.center) to (5);
		\draw [style=edge] (4.center) to (0);
		\draw [style=edge] (6) to (7);
		\draw [style=edge, >-] (3.center) to (7);
		\draw [style=edge] (7) to (0);
		\draw [style=edge, ->] (0) to (2.center);
		\draw [style=brace, pen colour=NiceRed, color=NiceRed] (9.center) to node [auto] {$p$} (8.center);
	\end{pgfonlayer}
\end{tikzpicture}
    \caption{$p = [x, g^{-1}x]$}
    \label{fig:translate2}
  \end{subfigure}%
  \par\bigskip
  \begin{subfigure}[b]{0.5\textwidth}
    \centering
    \begin{tikzpicture}
	\begin{pgfonlayer}{nodelayer}
		\node [style=vertex, label={below:$x$}] (0) at (-1, 0) {};
		\node [style=vertex, label={above:$\vo$}] (1) at (-1, 5) {};
		\node [style=none, label={right:$\Ax(g)$}] (2) at (8, 0) {};
		\node [style=none, label={left:$\Ax(g)$}] (3) at (-3, 0) {};
		\node [style=vertex, label={above:$gw$}] (5) at (6, 5) {};
		\node [style=vertex, label={above:$w$}] (6) at (1, 5) {};
		\node [style=vertex, label={below:$y$}] (7) at (1, 0) {};
		\node [style=vertex, label={below:$gy$}] (8) at (6, 0) {};
		\node [style=none] (9) at (-1, -1) {};
		\node [style=none] (10) at (1, -1) {};
		\node [style=none] (11) at (-1, -2) {};
		\node [style=none] (12) at (6, -2) {};
	\end{pgfonlayer}
	\begin{pgfonlayer}{edgelayer}
		\draw [style=edge] (6) to (7);
		\draw [style=edge] (0) to (1);
		\draw [style=edge] (8) to (5);
		\draw [style=edge, >-] (3.center) to (0);
		\draw [style=edge] (0) to (7);
		\draw [style=edge] (7) to (8);
		\draw [style=edge, ->] (8) to (2.center);
		\draw [style=brace, pen colour=NiceRed, color=NiceRed] (10.center) to node [auto] {$p$} (9.center);
		\draw [style=brace, pen colour=NiceBlue, color=NiceBlue] (12.center) to node [auto] {$q$} (11.center);
	\end{pgfonlayer}
\end{tikzpicture}
    \caption{$p = [x, y]$ and $g$ translates $y$ away from $x$}
    \label{fig:translate3}
  \end{subfigure}%
  \hfill
  \begin{subfigure}[b]{0.5\textwidth}
    \centering
    \begin{tikzpicture}
	\begin{pgfonlayer}{nodelayer}
		\node [style=vertex, label={below:$x$}] (0) at (2, 0) {};
		\node [style=vertex, label={above:$\vo$}] (1) at (2, 5) {};
		\node [style=none, label={right:$\Ax(g)$}] (2) at (6, 0) {};
		\node [style=none, label={left:$\Ax(g)$}] (3) at (-3, 0) {};
		\node [style=vertex, label={above:$gw$}] (5) at (4, 5) {};
		\node [style=vertex, label={above:$w$}] (6) at (-1, 5) {};
		\node [style=vertex, label={below:$y$}] (7) at (-1, 0) {};
		\node [style=vertex, label={below:$gy$}] (8) at (4, 0) {};
		\node [style=none] (9) at (-1, -1) {};
		\node [style=none] (10) at (2, -1) {};
		\node [style=none] (11) at (4, -1) {};
    \node [style=none] (12) at (-1, -2) {};
    \node [style=none] (13) at (6, -2) {};
	\end{pgfonlayer}
	\begin{pgfonlayer}{edgelayer}
		\draw [style=edge] (6) to (7);
		\draw [style=edge] (0) to (1);
		\draw [style=edge] (8) to (5);
		\draw [style=edge, >-] (3.center) to (7);
		\draw [style=edge] (7) to (0);
		\draw [style=edge] (0) to (8);
		\draw [style=edge, ->] (8) to (2.center);
		\draw [style=brace, pen colour=NiceRed, color=NiceRed] (10.center) to node[auto] {$p$} (9.center);
		\draw [style=brace, pen colour=NiceBlue, color=NiceBlue] (11.center) to node[auto] {$q$} (10.center);
    \draw [style=brace, opacity=0] (13.center) to node [auto] {$q$} (12.center);
	\end{pgfonlayer}
\end{tikzpicture}
    \caption{$p = [x, y]$ and $g$ translates $y$ towards $x$}
    \label{fig:translate4}
  \end{subfigure}
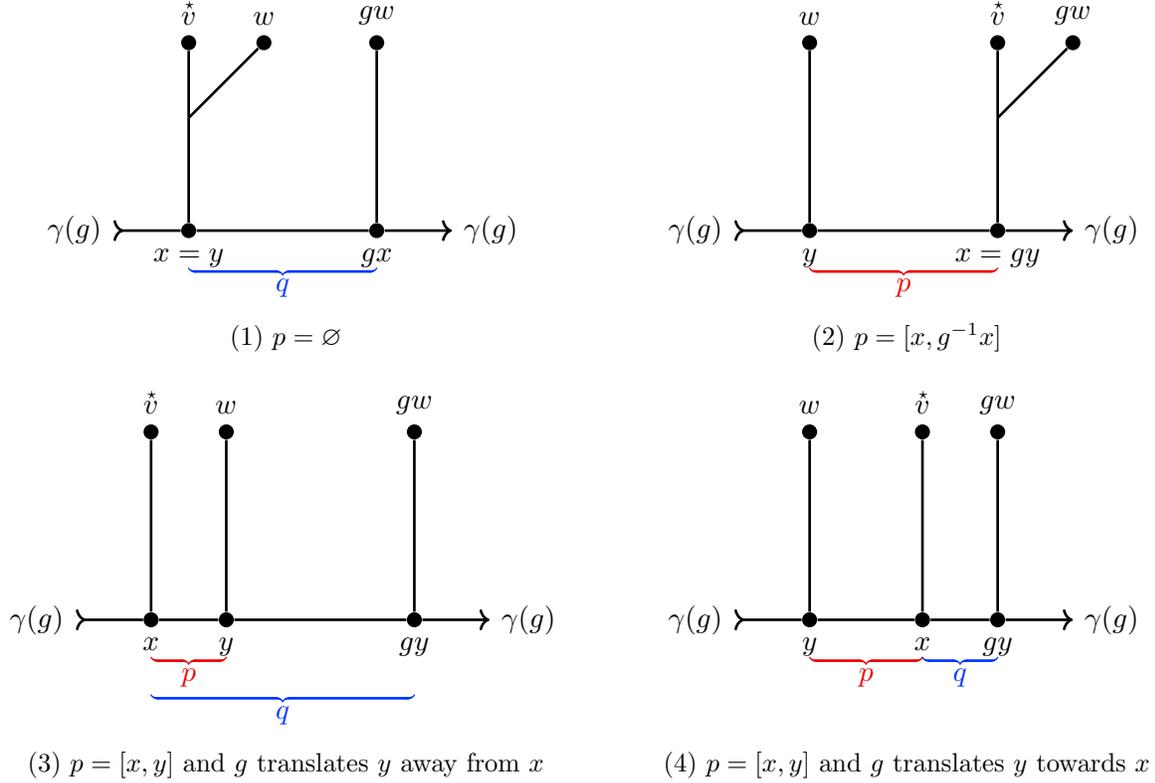
  \caption{The cases of Lemma \ref{lem:translate-overlap}}\label{fig:translate}
\end{midfigure}
\begin{proof}
  Let $P$ be the walk $[\vo, x] \oplus p \oplus [y, w]$, where $\oplus$ denotes concatenation. Let $Q$ be the walk $[\vo, x] \oplus q \oplus [gy, gw]$. We see that $\abs{Q} - \abs{P} = \abs{q} - \abs{p}$. Since $g$ translates $y$ along $\Ax(g)$ which is isometric to $\R$, we find that $\abs{q} = \abs[\big]{d(x, y) + a}$, where $a = l(g)$ if $g$ translates $y$ away from $x$, and $a = -l(g)$ if $g$ translates $y$ towards $x$. Therefore
  \[
    \abs{Q} = \abs{P} + \abs[\big]{d(x, y) + a} - d(x, y).
  \]
  Note that $\abs{P} \geq d(\vo, w)$ and $\abs{Q} \geq d(\vo, gw)$, with equality if and only if $P = [\vo, w]$ and $Q = [\vo, gw]$ respectively. Figure \ref{fig:translate} depicts each case. We achieve the formulae in cases (1)--(4) by the following substitutions:
  \begin{itemize}
    \item
    In case (1), $p = \varnothing$ so $P$ backtracks. Thus $\abs{P} > d(\vo, w)$, $Q = [\vo, gw]$, and $d(x, y) = 0$.
    
    \item
    In case (2), $q$ is either empty or a point. It follows that $\abs{Q} \geq d(\vo, gw)$, $P = [\vo, w]$, and $d(x, y ) = l(g) = -a$.
    
    \item
    In cases (3) and (4), $P = [\vo, w]$ and $Q = [\vo, gw]$. The cases are distinguished by whether $a$ is positive or negative. \qedhere
  \end{itemize}
\end{proof}
\begin{lemma}\label{lem:path-overlap-comparison}
  Let $w \in T$ and let $g \in \Aut(T)$ be hyperbolic. Suppose $[\vo, w] \cap \Ax(g)$ is a path $[x, y]$ for some distinct $x$ and $y$, and that $g$ translates $y$ towards $x$.
  \begin{itemize}
  \item If $d(x, y) > l(g)/2$, then $d(\vo, w) > d(\vo, gw)$.
  \item If $d(x, y) = l(g)/2$, then $d(\vo, w) = d(\vo, gw)$.
  \item If $d(x, y) < l(g)/2$, then $d(\vo, w) < d(\vo, gw)$.
  \end{itemize}
\end{lemma}
\begin{proof}
  If $d(x, y) > l(g)$, then by Lemma \ref{lem:translate-overlap} (4),
  \[
    d(\vo, gw) = d(\vo, w) + \abs{d(x, y) - l(g)} - d(x, y) = d(\vo, w) - l(g) < d(\vo, w).
  \]
  If $d(x, y) = l(g)$, then by Lemma \ref{lem:translate-overlap} (2), $d(\vo, gw) < d(\vo, w)$.
  If $d(x, y) < l(g)$, then by Lemma \ref{lem:translate-overlap} (4),
  \[
  d(\vo, gw) = d(\vo, w) + \abs{d(x, y) - l(g)} - d(x, y) = d(\vo, w) + l(g) - 2d(x, y).
  \]
  We see that $d(\vo, w) \leq d(\vo, gw)$ if and only if $d(x, y) \leq l(g)/2$, with equality when $d(x, y) = l(g)/2$.
\end{proof}
\begin{lemma}\label{lem:proj-is-intersect}
  Let $w \in T$ and let $g \in \Aut(T)$ be hyperbolic. The following are equivalent:
  \begin{enumerate}
    \item $\proj_{\Ax(g)}(w) \in U_g^\pm$.
    \item $[\vo, w] \cap U_g^\pm \neq \varnothing$.
    \item $[\vo, g^\e w] < [\vo, w]$ for some $\e \in \{1, -1\}$.
  \end{enumerate}
\end{lemma}
\begin{proof}
  \begin{enumerate}[wide, labelwidth=!, itemindent=!, labelindent=0pt]
    \item[(1) $\Rightarrow$ (2):] Suppose $y = \proj_{\Ax(g)}(w) \in U_g^\pm$. Let $x = \proj_{\Ax(g)}(\vo)$. Since $x \in U_g^0$, $[\vo, x]$ and $[y, w]$ are disjoint, so $[\vo, w] = [\vo, x] \cup [x, y] \cup [y, w]$.
    \item[(2) $\Rightarrow$ (3):] Suppose $[\vo, w] \cap U_g^\pm \neq \varnothing$. Then $[\vo, w] \cap \Ax(g)$ is a path $[x, y]$, where $x = \proj_{\Ax(g)}(\vo)$ and $y \in U_g^\pm$. It follows that $d(x, y) \geq l(g)/2$. We may assume that $g$ translates $y$ towards $x$. By Lemma \ref{lem:path-overlap-comparison}, $d(\vo, gw) \leq d(\vo, w)$. If $d(\vo, gw) = d(\vo, w)$, then $d(x, y) = l(g)/2$, so $y = u_g^-$. Therefore $[\vo, gw] < [\vo, w]$.
    
    \item[(3) $\Rightarrow$ (1):] Inverting $g$ if necessary, suppose $[\vo, gw] < [\vo, w]$. Now case (2) or (4) of Lemma \ref{lem:translate-overlap} must hold, that is $[\vo, w] \cap \Ax(g)$ is a path $[x, y]$ such that $g$ translates $y$ towards $x$. By Lemma \ref{lem:path-overlap-comparison}, $d(x, y) \geq l(g)/2$. If $d(x, y) = l(g)/2$, then $y \in \{u_g^+, u_g^-\}$ and $[\vo, gy] < [\vo, y]$, so $y = u_g^-$. In either case, $\proj_{\Ax(g)}(w) = y \in U_g^\pm$. \qedhere
  \end{enumerate}
\end{proof}
\begin{lemma}\label{lem:reduced-path}
  Suppose $X$ contains no elliptic elements and satisfies \textup{N1}. Then $X$ admits a fundamental system if and only if for all $g, h \in X^\pm$, if $g \neq h^{-1}$, then $[\vo, h\vo] < [\vo, gh\vo]$.
\end{lemma}
\begin{proof}
  Suppose $[\vo, gh\vo] \leq [\vo, h\vo]$ for some $g, h \in X^\pm$ with $g \neq h^{-1}$. If $[\vo, gh\vo] = [\vo, h\vo]$, then $h^{-1}gh\vo = \vo$, so $g$ is elliptic. Hence $[\vo, gh\vo] < [\vo, h\vo]$. Let $a = \proj_{\Ax(g)}(h\vo)$. By Lemma \ref{lem:proj-is-intersect}, $a \in [\vo, h\vo] \cap U_g^\pm$. Let $b \in [\vo, h\vo] \cap U_h^\pm$. We see that either $a \in [\vo, b]$ or $b \in [\vo, a]$. By Lemma \ref{lem:proj-is-intersect}, either $\proj_{\Ax(g)}(b) \in U_g^\pm$ or $\proj_{\Ax(h)}(a) \in U_h^\pm$. Therefore $X$ does not admit a fundamental system.
  
  Conversely, suppose $X$ does not admit a fundamental system. By Lemma \ref{lem:proj-is-intersect}, there exists distinct $g, h \in X$ and $a \in U_g^\pm$, $b \in U_h^\pm$ such that $a \in [\vo, b]$, for example as in Figure \ref{fig:fund-system-paths}. Inverting $g$ and $h$ if necessary, let
  \[
    a' \in [\vo, b] \cap [\vo, g\vo] \cap U_g^{\pm}, \qquad b' \in [\vo, b] \cap [\vo, h\vo] \cap U_h^\pm.
  \]
  Both $a'$ and $b'$ must exist, since both $[\vo, b]$ and $[\vo, g\vo]$ intersect the boundary of $U_g^\pm$, and similarly for $h$.
  We may assume (swapping $g$ and $h$ if necessary) that $a' \in [\vo, b']$. Then $a' \in [\vo, h\vo]$ so, by Lemma \ref{lem:proj-is-intersect}, $[\vo, g^\e h\vo] < [\vo, h\vo]$ for some $\e \in \{1, -1\}$.
\end{proof}
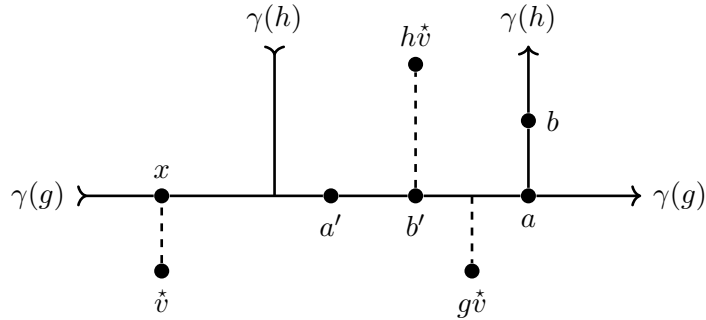
\begin{figure}[ht]
  \centering
  \begin{tikzpicture}[xscale=1.5]
	\begin{pgfonlayer}{nodelayer}
		\node [style=none, label={left:$\Ax(g)$}] (0) at (-5, 0) {};
		\node [style=none, label={right:$\Ax(g)$}] (1) at (5, 0) {};
		\node [style=vertex, label={below:$\vo$}] (2) at (-3.5, -2) {};
		\node [style=vertex, label={above:$x$}] (3) at (-3.5, 0) {};
		\node [style=none] (4) at (2, 0) {};
		\node [style=vertex, label={below:$g\vo$}] (5) at (2, -2) {};
		\node [style=none] (6) at (-1.5, 0) {};
		\node [style=none, label={above:$\Ax(h)$}] (7) at (-1.5, 4) {};
		\node [style=none, label={above:$\Ax(h)$}] (10) at (3, 4) {};
		\node [style=vertex, label={above:$h\vo$}] (11) at (1, 3.5) {};
		\node [style=vertex, label={below:$a$}] (12) at (3, 0) {};
		\node [style=vertex, label={right:$b$}] (13) at (3, 2) {};
		\node [style=vertex, label={below:$b'$}] (14) at (1, 0) {};
		\node [style=vertex, label={below:$a'$}] (16) at (-0.5, 0) {};
	\end{pgfonlayer}
	\begin{pgfonlayer}{edgelayer}
		\draw [style=edge, >-] (0.center) to (3);
		\draw [style=edge, ->] (12) to (1.center);
		\draw [style=edge-dash] (2) to (3);
		\draw [style=edge, -<] (6.center) to (7.center);
		\draw [style=edge, ->] (13) to (10.center);
		\draw [style=edge-dash] (14) to (11);
		\draw [style=edge-dash] (4.center) to (5);
		\draw [style=edge] (3) to (16);
		\draw [style=edge] (16) to (14);
		\draw [style=edge] (14) to (12);
    \draw [style=edge] (12) to (13);
	\end{pgfonlayer}
\end{tikzpicture}
  \caption{One construction for Lemma \ref{lem:reduced-path}}
  \label{fig:fund-system-paths}
\end{figure}
\begin{theorem}\label{thm:reduced-equivalent}
  $X$ is strongly \reduced if and only if $X$ admits a fundamental system.
\end{theorem}
\begin{proof}
  We may assume $X$ contains no elliptic elements and satisfies N1.
  Suppose $X$ is strongly \reduced. Let $g, h \in X^\pm$ such that $g \neq h^{-1}$. By Lemma \ref{lem:reduced-path}, it suffices to show that $[\vo, h\vo] < [\vo, gh\vo]$. This holds by definition if $\abs{h} < \abs{gh}$, so suppose otherwise. By N4, $\abs{h} = \abs{gh}$. We apply Lemma \ref{lem:same-initial-segment} to $h^{-1}$ and $g^{-1}$. From N2$'$ and this application we conclude that $\half(h^{-1}) = \half(h^{-1}g^{-1})$. By N4, $\half(h) < \half(gh)$. Therefore $[\vo, h\vo] < [\vo, gh\vo]$.
  
  Conversely, suppose $X$ is not strongly \reduced. If there exists $g, h \in X^{\pm}$ such that $g \neq h^{-1}$ and $\abs{gh} < \abs{h}$, then $[\vo, gh\vo] < [\vo, h\vo]$. Otherwise, suppose $X$ satisfies N2 and violates N4. Then there exists $g, h \in X^\pm$ such that $\abs{gh} = \abs{h}$ and $gh \prec h$. By a similar argument to the forward direction, $\half(gh) < \half(h)$, so $[\vo, gh\vo] < [\vo, h\vo]$. By Lemma \ref{lem:reduced-path}, $X$ does not admit a fundamental system.
\end{proof}
\begin{corollary}
  Every finitely generated purely hyperbolic subgroup of $\Aut(T)$ admits a unique fundamental system.
\end{corollary}
\begin{proof}
  Existence follows from Algorithm \ref{alg:reduce} and Theorem \ref{thm:reduced-equivalent}; uniqueness follows from Theorem \ref{thm:reduced-is-unique}.
\end{proof}
A \emph{fundamental domain} of $G$ is a subtree of $T$ containing exactly one vertex from every orbit of $G$ \cite[Chapter I.4]{serre-trees}.
\begin{definition}
  Suppose $G$ admits a fundamental system $U$. Define $\Gamma(G)$ to be the subtree of $T$ with vertex set
  \[
    \Set{w \in V(T) \given \proj_{\Ax(g)}(w) \in U^0_g \text{ for all } g \in X}.
  \]
\end{definition}
We show that $\Gamma(G)$ is a fundamental domain of $G$; see Corollary \ref{cor:fundamental-domain}.
\begin{lemma}\label{lem:outside-unique-axis}
  Suppose $X$ is strongly \reduced. For each vertex $w$ of $T$ there is at most one $g \in X$ such that $\proj_{\Ax(g)}(w) \notin U^0_g$.
\end{lemma}
\begin{proof}
  Let $g, h \in X$ be distinct. Let $a = \proj_{\Ax(g)}(w)$ and $b = \proj_{\Ax(h)}(w)$. Suppose $a \in U_g^\pm$ and $b \in U_h^\pm$. By Lemma \ref{lem:proj-is-intersect}, $a$ and $b$ are both in $[\vo, w]$, so either $a \in [\vo, b]$ or $b \in [\vo, a]$. By Lemma \ref{lem:proj-is-intersect} and Theorem \ref{thm:reduced-equivalent}, $X$ is not strongly \reduced.
\end{proof}
\begin{proposition}\label{prop:fund-domain-1}
  Suppose $X$ is strongly \reduced. Let $w \in V(T)$. The following are equivalent:
  \begin{enumerate}
    \item $w \in \Gamma(G)$. 
    \item $[\vo, w] \ltp [\vo, xw]$ for all $x \in X^\pm$.
    \item $[\vo, w] \ltp [\vo, gw]$ for all $g \in G \setminus \{1\}$.
  \end{enumerate}
\end{proposition}
\begin{proof}
  The equivalence of $(1)$ and $(2)$ is given directly by Lemma \ref{lem:proj-is-intersect}, and (3) $\Rightarrow$ (2) is trivial. We now prove $(2) \Rightarrow (3)$. Suppose $[\vo, w] \ltp [\vo, xw]$ for all $x \in X^\pm$, or equivalently $w \in \Gamma(G)$. Let $g = x_n^{\e_n} \dots x_1^{\e_1}$, where each $x_i \in X$, $\e_i \in \Z \setminus \{0\}$, and $x_i \neq x_{i+1}$. Let $h = x_{n-1}^{\e_{n-1}} \dots x_{1}^{\e_1}$ and if $n \geq 2$, then let $h' = x_{n-2}^{\e_{n-2}} \dots x_{1}^{\e_1}$. By transitivity, it suffices to show that $[\vo, hw] < [\vo, gw]$. We proceed by induction on $n$.
  
  First we show that $\proj_{\Ax(x_n)}(hw) \in U_{x_n}^0$.
  If $n=1$, then $h$ is trivial, so the statement holds by Lemma \ref{lem:proj-is-intersect}. Suppose $n > 1$.
  Note that $U_{x_i}^0 \subseteq U_{x_i^{\e_i}}^0$ and $U_{x_i^{\e_i}}^\pm \subseteq U_{x_i}^\pm$ for all $i$, since $\Ax(x_i) = \Ax(x_i^{\e_i})$ and $U_{x_i^{\e_i}}^0$ has length $l(x_i^{\e_i}) \geq l(x_i)$.
  Additionally, by the induction hypothesis $[\vo, h'w] < [\vo, hw]$. By Lemma \ref{lem:proj-is-intersect}, 
  \[
    \proj_{\Ax(x_{n-1})}(hw) \in U_{x_{n-1}^{\e_{n-1}}}^\pm \subset U_{x_{n-1}}^\pm.
  \]
  By Lemma \ref{lem:outside-unique-axis}, $\proj_{\Ax(x_n)}(hw) \in U_{x_n}^0$, as desired.
  
  It follows that $\proj_{\Ax(x_n)}(hw) \in U_{x_n^{\e_n}}^0$. By Lemma \ref{lem:proj-is-intersect}, $[\vo, hw] < [\vo, gw]$, completing the proof.
\end{proof}
\begin{corollary}\label{cor:fundamental-domain}
  $\Gamma(G)$ has vertex set $\Set{w \in V(T) \given [\vo, w] \ltp [\vo, gw] \text{ for all } g \in G \setminus \{1\}}$. In particular, $\Gamma(G)$ is a fundamental domain of $G$. 
\end{corollary}
\begin{proof}
  Let $u \in V(T)$. There must be some unique $g \in G$ minimising $[\vo, gu]$, since $\ltp$ is a well-ordering. By Proposition \ref{prop:fund-domain-1}, $gu$ is the unique element of $Gu$ in $\Gamma(G)$. 
\end{proof}
\begin{remark}
  $\Gamma(G)$ is closely related to the \emph{Dirichlet fundamental domain} \cite[Definition 1.8]{lubotzky}. However, a Dirichlet fundamental domain is not necessarily a fundamental domain, since it also contains each vertex $w$ of $T$ such that $\proj_{\Ax(g)}(w) = u_g^+$ for some $g \in X$.
\end{remark}
Given $w \in V(T)$, Algorithm \ref{alg:fundamental} finds the unique $g \in G$ such that $gw \in \Gamma(G)$.
\begin{algorithm}[ht]
  \DontPrintSemicolon
  \KwData{Finite strongly \reduced subset $X$ of $\Aut(T)$, $w \in V(T)$}
  \KwOut{The unique $g \in G=\gen{X}$ such that $gw \in \Gamma(G)$}
  $g \gets 1$\;
  \While{there exists $x \in X^{\pm}$ such that $[\vo, xgw] \ltp [\vo, gw]$}{
    $g \gets xg$\;
  }
  \Return{$g$}\;
  \caption{Find $g \in G$ mapping vertex to fundamental domain}
  \label{alg:fundamental}
\end{algorithm}
\begin{proof}[Proof of correctness of Algorithm \emph{\ref{alg:fundamental}}]
  Let $a = \proj_{\Ax(g)}(w)$.
  Suppose $w \notin \Gamma(G)$. By Proposition \ref{prop:fund-domain-1}, there is some $x \in X^\pm$ such that $[\vo, xw] < [\vo, w]$. We thus obtain a sequence $w = w_0, w_1 \dots$, where $w_i = x_i w_{i-1}$ for some $x_i \in X^\pm$ and $[\vo, w_i] < [\vo, w_{i-1}]$ for each $i$. Eventually this sequence must terminate, so we find some $w_n = x_n \dots x_1 w$ such that $[\vo, w_n] < [\vo, xw_n]$ for all $x \in X^\pm$. By Proposition \ref{prop:fund-domain-1}, $w_n \in \Gamma(G)$.
\end{proof}
As a corollary, we obtain the following:
\conMem*
\begin{proof}
  Let $g \in \Aut(T)$. We wish to decide whether $g \in G$, and if so write $g$ as a word in $X$.
  We may assume $X$ is strongly \reduced, since by keeping track of the transformations involved in Algorithm \ref{alg:reduce} we may write each word in a strongly \reduced basis of $G$ as a word in $X$.
  Using Algorithm \ref{alg:fundamental}, we find the unique $h \in G$, and a reduced word of $h$ in $X$, such that $hg\vo \in \Gamma(G)$. It follows that $g \in G$ if and only if $g = h^{-1}$.
\end{proof}
\begin{remark}
  The proof of Theorem \ref{thm:constructive-membership} provides an alternative algorithm for Corollary \ref{cor:equal-subgroup}: If $X$ and $Y$ are strongly \reduced, then $\gen{X} = \gen{Y}$ if and only if $x \in \gen{Y}$ and $y \in \gen{X}$ for all $x \in X$ and $y \in Y$. However, we included the first algorithm to show that deciding equality of these groups is easier than Theorem \ref{thm:constructive-membership} would imply.
\end{remark}

\section{Implementation and performance}\label{sec:implementation}

We implemented in \textsc{Magma} \cite{magma} our algorithms for $\PGL_2(K)$, where $K$ is a $p$-adic field, acting on the Bruhat-Tits tree \cite{magmacode}. This package implements Algorithms \ref{alg:reduce}, \ref{alg:fundamental}, and the algorithms given in the proofs of Corollary \ref{cor:equal-subgroup} and Theorem \ref{thm:constructive-membership}. In our implementation, $X$ is input as a sequence rather than a set, since this simplifies the code.

Table \ref{table:runtime-generators} shows the runtime of Algorithm \ref{alg:reduce}, and the average time per iteration of the loop, for different values of $\abs{X}$. The trials are run with randomly chosen elements of $\SL_2(\Q_5)$ with 1000 digits of precision, where each entry has valuation at most 10. Changing the choice of prime $p$ and the precision did not seem to significantly affect the runtime. The times shown are averaged over 1000 trials.

Increasing the maximum valuation of entries of the generators rapidly leads to the algorithm losing precision, so we do not run trials over a range of maximum valuations.

Tables \ref{table:runtime-fundamental-gens} and \ref{table:runtime-fundamental-dist} show runtimes of Algorithm \ref{alg:fundamental}. Recall that Algorithm \ref{alg:fundamental} takes as input a strongly \reduced subset of $\Aut(T)$ and a vertex of $T$. We record the average runtime over all pairings of 100 strongly \reduced generating sets and 100 vertices.
For Table \ref{table:runtime-fundamental-gens} we vary $\abs{X}$ and choose $w$ such that $d(w, \vo) = 100$, and for
Table \ref{table:runtime-fundamental-dist} we vary $d(w, \vo)$ and set $\abs{X} = 5$. The code used for these tests is provided in \cite{magmacode}. The trials were run using \textsc{Magma} V2.28-2 on a \qty{2.6}{GHz} machine.
\vfill
\begin{table}[ht]
  \centering
  \begin{tabular}{c|c|c}
    \hline
    $\abs{X}$ & Average time (seconds) & Average time / iteration (seconds) \\
    \hline
    2 & 0.03 & 0.01 \\
    5 & 0.80 & 0.07 \\
    10 & 3.06 & 0.18 \\
    100 & 10.87 & 0.97 \\
    \hline
  \end{tabular}
  \caption{Runtimes of Algorithm \ref{alg:reduce} for different numbers of generators}
  \label{table:runtime-generators}
\end{table}
\vfill
\begin{table}[ht]
  \centering
  \begin{tabular}{c|c}
    \hline
    $\abs{X}$ & Average time (seconds) \\
    \hline
    2 & 0.04 \\
    3 & 0.06 \\
    5 & 0.10 \\
    10 & 0.21 \\
    \hline
  \end{tabular}
  \caption{Runtimes of Algorithm \ref{alg:fundamental} for different numbers of generators}
  \label{table:runtime-fundamental-gens}
\end{table}
\vfill
\begin{table}[ht]
  \centering
  \begin{tabular}{c|c}
    \hline
    $d(w, \vo)$ & Average time (seconds) \\
    \hline
    5 & 0.01 \\
    10 & 0.02 \\
    20 & 0.03 \\
    50 & 0.05 \\
    100 & 0.10 \\
    \hline
  \end{tabular}
  \caption{Runtimes of Algorithm \ref{alg:fundamental} for various distances of $w$ from $\vo$}
  \label{table:runtime-fundamental-dist}
\end{table}
\vfill
\pagebreak




\bibliography{references}
\bibliographystyle{abbrv}

\end{document}